\documentclass[11pt,a4paper]{amsart}
\usepackage[utf8]{inputenc}
\usepackage{amssymb}
\usepackage{amsmath}
\usepackage{amsthm}
\usepackage{xcolor}
\usepackage{hyperref}
\usepackage{tikz-cd}
\usepackage{mathscinet}
\usepackage[top=3 cm, bottom=3 cm, left=3 cm, right=3cm]{geometry}
\usepackage{hyperref}
    \hypersetup{colorlinks=true,allcolors=red}
    \usepackage[capitalize,nameinlink,noabbrev]{cleveref}

\usepackage[backref=true,giveninits=true,doi=false,url=false,isbn=false,backend=biber]{biblatex}
    \newbibmacro{string+doiurlisbn}[1]{%
      \iffieldundef{doi}{%
        \iffieldundef{url}{%
          \iffieldundef{isbn}{%
            \iffieldundef{issn}{%
              #1%
            }{%
              \href{http://books.google.com/books?vid=ISSN\thefield{issn}}{#1}%
            }%
          }{%
            \href{http://books.google.com/books?vid=ISBN\thefield{isbn}}{#1}%
          }%
        }{%
          \href{\thefield{url}}{#1}%
        }%
      }{%
        \href{http://dx.doi.org/\thefield{doi}}{#1}%
      }%
    }
    \DeclareFieldFormat{title}{\usebibmacro{string+doiurlisbn}{\mkbibemph{#1}}}
    \DeclareFieldFormat[article,incollection,thesis,misc,inproceedings,online,inbook]{title}{\usebibmacro{string+doiurlisbn}{\mkbibquote{#1}}}
    \DeclareFieldFormat{year}{\usebibmacro{year-parenthesis}{\mkbibemph{#1}}}
\renewbibmacro*{date}{%
\iffieldundef{year}
{}
{\ifentrytype{misc}{\printtext[parens]{\printdate}}{\printdate}}}%
\hypersetup{
           breaklinks=true,   
           colorlinks=true,   
           pdfusetitle=true,  
        }

\newcounter{dummy}
\numberwithin{dummy}{section}
\newtheorem{thm}[dummy]{Theorem}

\newtheorem{lem}[dummy]{Lemma}
\newtheorem{prop}[dummy]{Proposition}
\newtheorem{cor}[dummy]{Corollary}
\theoremstyle{definition}
\newtheorem{rmk}[dummy]{Remark}
\newtheorem{example}[dummy]{Example}
\numberwithin{equation}{section}

\newcommand{\noop}[1]{}

\DeclareMathOperator{\End}{End}

\begin{document}

\title{Around the support problem for Hilbert class polynomials}
\date{\today}

\author[F. Campagna]{Francesco Campagna}
\address{Max-Planck-Institut für Mathematik, Vivatsgasse 7, 53111 Bonn, Germany}
\email{campagna@mpim-bonn.mpg.de}

\author[G. A. Dill]{Gabriel A. Dill}
\address{Institut f\"ur Algebra, Zahlentheorie und Diskrete Mathematik, Fakult\"at f\"ur Mathematik und Physik, Leibniz Universit\"at Hannover, Welfengarten 1, 30167 Hannover, Germany}
\email{dill@math.uni-hannover.de}

\date{\today}
\maketitle

\begin{abstract}
Let $H_D(T)$ denote the Hilbert class polynomial of the imaginary quadratic order of discriminant $D$. We study the rate of growth of the greatest common divisor of $H_D(a)$ and $H_D(b)$ as $|D| \to \infty$ for $a$ and $b$ belonging to various Dedekind domains. We also study the modular support problem: if for all but finitely many $D$ every prime ideal dividing $H_D(a)$ also divides $H_D(b)$, what can we say about $a$ and $b$? If we replace $H_D(T)$ by $T^n-1$ and the Dedekind domain is a ring of $S$-integers in some number field, then these are classical questions that have been investigated by Bugeaud--Corvaja--Zannier, Corvaja--Zannier, and Corrales-Rodrig\'a\~{n}ez--Schoof.
\end{abstract}

\section{Introduction}

The starting point of this paper is the following theorem by Bugeaud, Corvaja, and Zannier.

\begin{thm}[Theorem 1 in \cite{Bugeaud_Corvaja_Zannier_2003}] \label{thm:BCZ}
Let $a,b$ be multiplicatively independent integers $\geq 2$, and let $\varepsilon>0$. Then, provided $n$ is sufficiently large, we have
\[
\gcd (a^n-1, b^n-1) < \mathrm{exp}(\varepsilon n).
\]
\end{thm}

Here, the left-hand side measures the size of the intersection of the Zariski closure in $\mathbb{G}^2_{m,\mathbb{Z}}$ of the singleton $\{(a,b)\} \subseteq \mathbb{G}_{m,\mathbb{Q}}^{2}(\mathbb{Q})$ with the kernel of the raising-to-the-$n$-th-power morphism. This is an intersection of two $1$-dimensional schemes inside a $3$-dimensional scheme. Thus, Theorem \ref{thm:BCZ} fits into the framework of  ``unlikely intersections" as introduced by Bombieri, Masser, and Zannier \cite{BMZ99}, Zilber \cite{Zilber}, and Pink \cite{PinkUnpubl}. As remarked by Zannier in \cite[Chapter 2]{Zannier_book}, Theorem \ref{thm:BCZ} can be regarded as an arithmetical analogue of results about unlikely intersections over fields of characteristic $0$. Silverman conjectured in \cite[Conjecture 1]{Silverman_2004} an analogue of Theorem \ref{thm:BCZ}, where the square of the multiplicative group is replaced by the square of an elliptic curve over the rationals. Assuming certain cases of Vojta's conjecture, Silverman proved a generalization of \cite[Conjecture 1]{Silverman_2004} to an arbitrary abelian variety over $\mathbb{Q}$ in \cite[Proposition 9]{Silverman_2005}. Unconditionally, nothing non-trivial is known at the time of writing in the setting of abelian varieties over number fields.

In this article, we replace $\mathbb{G}_m$ by the coarse moduli space of elliptic curves $Y(1) = \mathbb{A}^1$ and we study the analogue of Theorem \ref{thm:BCZ} and related questions in this context. This venture is inspired by the well-known fact that there is a notion of special subvarieties in both the realm of semiabelian varieties and the realm of mixed Shimura varieties. Let $F$ be a field. A \emph{special subvariety} of $\mathbb{G}^n_{m,F}$ is an irreducible component of an algebraic subgroup of $\mathbb{G}^n_{m,F}$. If $X_1,\hdots,X_n$ are affine coordinates on $Y(1)_F^n$, then a \emph{special subvariety} of $Y(1)_F^n$ is an irreducible component of the intersection of the zero loci of finitely many modular polynomials $\Phi_{N_k}(X_{i_k},X_{j_k})$ ($k = 1,\hdots,K$); see \cite[p. 55]{Lang_1987} for the definition of $\Phi_{N_k}$. In particular, a special point of $\mathbb{G}_{m,\mathbb{C}}$ is a root of unity and a special point of $Y(1)_{\mathbb{C}}$ is a \emph{singular modulus}, \emph{i.e}., the $j$-invariant of an elliptic curve with complex multiplication.

Theorem \ref{thm:BCZ} is about values of the polynomials $T^n-1$ ($n \in \mathbb{N} = \{1,2,\hdots\}$). These have the property that their zeroes are all special points of $\mathbb{G}_{m,\mathbb{C}}$. It then appears more natural to consider the family of cyclotomic polynomials $\Psi_n(T)$ ($n \in \mathbb{N}$), which are precisely the minimal polynomials over $\mathbb{Q}$ of the special points of $\mathbb{G}_{m,\mathbb{C}}$. Using the dictionary above, the analogue of this family in the $Y(1)$ case is precisely the family of \emph{Hilbert class polynomials} $H_D(T)$ with $D \in \mathbb{D}$, where $\mathbb{D} = \{-3,-4,\hdots\}$ is the set of negative integers $\equiv 0,1~\mathrm{mod}~4$ and $H_D(T) \in \mathbb{Z}[T]$ is the minimal polynomial over $\mathbb{Q}$ of any $j$-invariant of an elliptic curve with complex multiplication by the imaginary quadratic order of discriminant $D$. Thus, we are led to studying how large the greatest common divisor of $H_D(a)$ and $H_D(b)$ can be, where $a, b \in \mathbb{Z}$.

This question as well as more general divisibility questions also make sense with an arbitrary Dedekind domain $R$ in place of $\mathbb{Z}$. Indeed, for any polynomial with integer coefficients like for example $\Psi_n(T)$, $H_D(T)$, and $\Phi_N(X,Y)$, we obtain an associated polynomial with coefficients in $R$ by applying the unique ring homomorphism $\mathbb{Z} \to R$ to all coefficients. We will always denote this associated polynomial by the same symbol as the original one and also speak of Hilbert class polynomials, cyclotomic polynomials, modular polynomials, etc. with coefficients in $R$.

In Section \ref{sec:functionfield}, we begin by studying the case where $R$ is the coordinate ring of a smooth affine irreducible curve $\mathcal{C}$ over an algebraically closed field of characteristic $0$. We obtain the following result.

\begin{thm}[Theorem \ref{thm: gcd in function fields}]\label{thm: gcd in function fields intro}
Let $F$ be an algebraically closed field of characteristic $0$, let $R$ be the coordinate ring of a smooth affine irreducible curve $\mathcal{C}_{/F}$, and let $A,B \in R$. If $\Phi_N(A,B) \neq 0$ for all $N \in \mathbb{N}$, then there exists a non-zero ideal $J \subseteq R$ such that
\[ \gcd(H_{D_1}(A),H_{D_2}(B)) \mid J \quad \mbox{ or } \quad H_{D_1}(A)H_{D_2}(B) = 0  \]
for all $D_1, D_2 \in \mathbb{D}$.
\end{thm}

For $T^n-1$ in place of $H_D(T)$ and $\mathcal{C} = \mathbb{A}^1_F$, Ailon and Rudnick have proved the analogue of Theorem \ref{thm: gcd in function fields intro} as the first part of \cite[Theorem 1]{AilonRudnick} using the Manin-Mumford conjecture for plane curves, \emph{i.e.},  the theorem of Ihara--Serre--Tate \cite{Lang_1965}. Our result similarly follows from the Andr\'e-Oort conjecture for plane curves, \emph{i.e.}, from the theorem of Andr\'e \cite{Andre}, which is the modular counterpart of the theorem of Ihara--Serre--Tate. Silverman also studied the analogue of this question in the elliptic setting, but managed to obtain a comparable result only if the elliptic curve is isotrivial, see \cite[Conjecture 7, Theorem 8, and Remark 5]{Silverman_2004}. One may wonder whether a similar theorem also holds if the  characteristic of $F$ is positive. Our next theorem, which we prove using the $abc$ theorem for function fields by Mason \cite{Mason}, shows that this is not the case if $F$ is an algebraic closure of a finite field.

\begin{thm}[Theorem \ref{thm: gcd singular moduli positive characteristic}]\label{thm:gcd singular moduli char p intro}
Let $p \in \mathbb{N}$ be prime and fix an algebraic closure $F = \overline{\mathbb{F}}_p$ of $\mathbb{F}_p$. Let $R$ be the coordinate ring of a smooth affine irreducible curve $\mathcal{C}_{/F}$ and let $A, B \in R\backslash F$. Then
\[
\limsup_{D \in \mathbb{D},\hspace{2pt}|D| \to \infty} \frac{\deg \left(\gcd(H_D(A), H_D(B)) \right)}{\deg H_D} > 0,
\]
where $\deg$ is as defined in Subsection \ref{subsec:deg}.
\end{thm}

In the multiplicative setting, Silverman proved a more precise analogue of Theorem \ref{thm:gcd singular moduli char p intro} as \cite[Theorem 4]{Silverman_2004b} in the case where $\mathcal{C} = \mathbb{A}^1_F$. He also considered the elliptic case for the same $\mathcal{C}$, but again obtained a comparable result only under the added condition of isotriviality, see \cite[Conjecture 9 and Theorem 10]{Silverman_2004}. A related question has been studied in the modular setting by Edixhoven and Richard in \cite{Edixhoven_Richard_2018}, see also Richard's work \cite{Richard_Preprint_2018} over $\mathbb{Z}$.

We finally come to the modular counterpart of Theorem \ref{thm:BCZ} and its generalization to arbitrary number fields.

\begin{thm}[Theorem \ref{thm: BCZ_fails}]\label{thm: BCZ_fails intro}
Let $K$ be a number field and let $S$ be a finite set of maximal ideals of $\mathcal{O}_K$. Consider two elliptic curves ${E_1}_{/K}, {E_2}_{/K}$ with potential good reduction outside of $S$, \emph{i.e.}, such that the $j$-invariant $j(E_i) \in \mathcal{O}_{K,S}$ for $i = 1,2$. Suppose that there exists a prime ideal $\mathfrak{p} \subseteq \mathcal{O}_{K,S}$ at which both $E_1$ and $E_2$ have potential good supersingular reduction. Let $p$ denote the rational prime lying under $\mathfrak{p}$. Then

\[ \limsup_{D \in \mathbb{D},\hspace{2pt}|D| \to \infty}{(\deg H_D)^{-1}\log N(\gcd(H_D(j(E_1)),H_D(j(E_2))))} \geq \frac{\log p}{p-1} > 0,\]

where $N(\cdot)$ denotes the ideal norm in $\mathcal{O}_{K,S}$.
\end{thm}

In particular, Theorem \ref{thm: BCZ_fails intro} shows that the naive analogue of Theorem \ref{thm:BCZ} over number fields is false. Namely, the condition in Theorem \ref{thm:BCZ} that $a$ and $b$ are multiplicatively independent is equivalent to demanding that the point $(a,b) \in \mathbb{G}_{m,\mathbb{C}}^2(\mathbb{C})$ is not contained in any proper special subvariety of $\mathbb{G}^2_{m,\mathbb{C}}$. In the modular setting, this translates into the condition that $a$ and $b$ are both not singular moduli and that $\Phi_N(a,b) \neq 0$ for all $N \in \mathbb{N}$. However, we will show that there are infinitely many such pairs $(a,b) \in \mathcal{O}_{K,S}^2$ such that $a$ and $b$ are the $j$-invariants of elliptic curves with a fixed common prime ideal of potential good supersingular reduction. On the other hand, proving that two general elliptic curves have a common supersingular prime is a difficult open problem, already if they are defined over $\mathbb{Q}$. Conjecturally, there are infinitely many common supersingular primes if both $E_1$ and $E_2$ do not have complex multiplication, are defined over $\mathbb{Q}$, and are not geometrically isogenous; see \cite{LangTrotter} and \cite{Fouvry_Murty_1995}, where also an averaged version of this conjecture is proved.

There are other examples where the arithmetic behaviour in Shimura varieties and in algebraic groups is different. Let us mention, for example, the $S$-integrality properties of special points with respect to a divisor in $\mathbb{G}_{m,\overline{\mathbb{Q}}}$ and $Y(1)_{\overline{\mathbb{Q}}}$ respectively, where $S$ is a finite set of rational primes and $\overline{\mathbb{Q}}$ will always denote a fixed algebraic closure of $\mathbb{Q}$: for $\mathbb{G}_{m,\overline{\mathbb{Q}}}$, Baker, Ih, and Rumely proved in \cite[Theorem~0.1]{BIR_2008} that there are at most finitely many special points that are $S$-integral with respect to some non-special point $P$, but this fails to hold if $P$ itself is special. For $Y(1)_{\overline{\mathbb{Q}}}$, Habegger proved in \cite[Theorem 2]{Habegger_2015} that there are at most finitely many special points that are $\emptyset$-integral with respect to an \emph{arbitrary} finite non-empty set of points, which can also consist of only one special point.

In order to prove Theorem \ref{thm: BCZ_fails intro}, we apply \cite[Theorem 3]{Michel_2004} by Michel to find arbitrarily large discriminants $D \in \mathbb{D}$ such that many zeroes of $H_D(T)$ reduce to the reductions of $j(E_1)$ and $j(E_2)$ respectively modulo a fixed prime ideal $\mathfrak{P}$ that lies over $\mathfrak{p}$ in a fixed algebraic closure of $K$. For our argument to work, it is essential that each such $D$ is coprime to $p$ and that therefore $\mathfrak{p}$ ramifies only very little in the splitting field of $H_D(T)$ over $K$.

Supersingularity seems to be a fundamental feature of the modular world that the multiplicative world is lacking and that explains some of the differences between the two worlds: it allows several distinct special points in the same Galois orbit to reduce to the same element modulo a maximal ideal lying over some rational prime $p$ while $p$ remains unramified in the corresponding ring class field. Multiplying the discriminant of a singular modulus or the order of a root of unity by powers of $p$ also leads to several distinct elements of the Galois orbit having equal reductions modulo a maximal ideal above $p$ (see Proposition \ref{prop:Hilbert_class_polynomial_mod_p} for the modular case), but at the same time introduces arbitrarily large ramification over $p$. It seems likely that the supersingular primes are the only obstacle to proving an analogue of Theorem \ref{thm:BCZ} in the modular case. In particular, we can ask for the rate of growth of the norm of the g.c.d. of $H_{D_1}(a)$ and $H_{D_2}(b)$ deprived of all common supersingular prime factors for $a$, $b$ in some ring of $S$-integers such that $\Phi_N(a,b) \neq 0$ for all $N \in \mathbb{N}$ and neither $a$ nor $b$ is a singular modulus.

When trying to understand the size of the greatest common divisor of $H_D(a)$ and $H_D(b)$ for $a$ and $b$ in some Dedekind domain $R$, we were led to consider the extreme case where every prime dividing $H_D(a)$ also divides $H_D(b)$ for all but finitely many $D \in \mathbb{D}$. For which $a$ and $b$ is this possible? This is the modular instance of the so-called \emph{support problem}. If we again replace the polynomials $H_D(T)$ ($D \in \mathbb{D}$) by the polynomials $T^n-1$ ($n \in \mathbb{N}$), this becomes the \emph{multiplicative support problem}. In the case where $R$ is the ring of $S$-integers in some number field, Corrales-Rodrig\'a\~{n}ez and Schoof have solved this problem in \cite{Corrales-Rodriganez_Schoof_1997}, thereby answering a question by Erd\H{o}s at the $1988$ number theory conference in Banff. Corrales-Rodrig\'a\~{n}ez and Schoof also solved an analogue of this problem with an elliptic curve in place of the multiplicative group, see \cite[Theorem 2]{Corrales-Rodriganez_Schoof_1997}. Later, Larsen solved the problem in the case of an arbitrary abelian variety \cite{Larsen_2003} and Perucca generalized his result to split semiabelian varieties \cite{Perucca_2009}. It is interesting to note that the conclusions of all these results are invariant under localization of the base ring at some finite set of maximal ideals although their hypotheses depend on the base ring. This will also be the case in our own results. One can also investigate the support problem as well as possible analogues of Theorem \ref{thm:BCZ} in other Shimura varieties than $Y(1)^2$ or in arithmetic dynamics (for the latter, cf. \cite[Section 5]{Matsuzawa_Preprint}).

In Section \ref{sec:supportproblem}, we revisit the multiplicative support problem as well as its cyclotomic counterpart, where $T^n-1$ is replaced by $\Psi_n(T)$ for $n \in \mathbb{N}$. We show how the theorem of Corrales-Rodrig\'a\~{n}ez--Schoof \cite{Corrales-Rodriganez_Schoof_1997} and the theorem of Ihara--Serre--Tate \cite{Lang_1965} can be used to give the following comprehensive solution to the multiplicative and cyclotomic support problem in characteristic $0$, which is best possible in all cases.

\begin{thm}\label{thm:multiplicative_and_cyclotomic_support_problem_char_0}
Fix $N_0 \in \mathbb{N}$.

Let $K$ be a number field, let $S$ be a finite set of maximal ideals of $\mathcal{O}_K$, and let $a,b \in \mathcal{O}_{K,S}\backslash\{0\}$. Then the following hold:
\begin{enumerate}
    \item Suppose that for all $n \in \mathbb{N}$ with $n > N_0$, every prime ideal of $\mathcal{O}_{K,S}$ that divides $a^n-1$ also divides $b^n-1$. Then $b = a^k$ for some $k \in \mathbb{Z}$.
    \item Suppose that for all $n \in \mathbb{N}$ with $n > N_0$, every prime ideal of $\mathcal{O}_{K,S}$ that divides $\Psi_n(a)$ also divides $\Psi_n(b)$. Then either $a$ and $b$ are roots of unity of the same order or $b = a^{\pm 1}$.
\end{enumerate}

Let $R$ be the coordinate ring of a smooth affine irreducible curve $\mathcal{C}$ over an algebraically closed field $F$ of characteristic $0$, and let $A, B \in R \backslash F$. Then the following hold:

\begin{enumerate}
\setcounter{enumi}{2} 
    \item Suppose that for all $n \in \mathbb{N}$ with $n > N_0$, every prime ideal of $R$ that divides $A^n-1$ also divides $B^n-1$. Then $B = A^k$ for some $k \in \mathbb{Z}\backslash\{0\}$.
    \item Suppose that for all $n \in \mathbb{N}$ with $n > N_0$, every prime ideal of $R$ that divides $\Psi_n(A)$ also divides $\Psi_n(B)$. Then $B = A^{\pm 1}$.
\end{enumerate}
\end{thm}

Part (1) of Theorem \ref{thm:multiplicative_and_cyclotomic_support_problem_char_0} is the theorem of Corrales-Rodrig\'a\~{n}ez--Schoof \cite{Corrales-Rodriganez_Schoof_1997}, the other parts are the contents of Theorem \ref{thm: cyclotomic_support_problem} and Corollary \ref{thm:g_m_multiplicative_support_problem_function_field_char_0}.

In Section \ref{sec:modularsupportproblem}, we finally consider the \emph{modular support problem}. We first solve the function field version of the problem.

\begin{thm}[Theorem \ref{thm: support problem in function fields}]\label{thm: support problem in function fields intro}
Let $R$ be the coordinate ring of a smooth affine irreducible curve $\mathcal{C}$ over an algebraically closed field $F$ of characteristic $0$. Let $A, B \in R \backslash F$ and suppose that there exists $D_0 \in \mathbb{N}$ such that for all discriminants $D \in \mathbb{D}$ with $|D|>D_0$ every prime ideal of $R$ that divides $H_D(A)$ also divides $H_D(B)$. Then $A=B$.
\end{thm}

In the proof of Theorem \ref{thm: support problem in function fields intro}, we rely on Theorem \ref{thm: gcd in function fields intro}, but additional arguments are necessary to deal with the case where $\Phi_N(A,B) = 0$ for some $N \in \mathbb{N}$, $N > 1$. We then turn to the modular support problem in the number field case and prove the following theorem.

\begin{thm}[Theorem \ref{thm:modular_support_problem}]\label{thm:modular_support_problem intro}
Let $K$ be a number field and let $S$ be a finite set of maximal ideals of $\mathcal{O}_K$. Let $j, j' \in \mathcal{O}_{K,S}$. Suppose that there exists $D_0 \in \mathbb{N}$ such that all the prime ideals of $\mathcal{O}_{K,S}$ dividing $H_D(j)$ also divide $H_D(j')$ for every $D \in \mathbb{D}$ with $|D| > D_0$. Then either $j = j'$ or there exists $\widetilde{D} \in \mathbb{D}$ such that $H_{\widetilde{D}}(j)=H_{\widetilde{D}}(j')=0$.
\end{thm}

An important ingredient in our proof of Theorem \ref{thm:modular_support_problem intro} is a result of Zarhin \cite{Zarhin_2017} that allows us to find many primes of good ordinary reduction for a given elliptic curve without complex multiplication such that the endomorphism ring of its reduction satisfies some suitable local conditions. We also crucially use a result by Khare and Larsen \cite{KhareLarsen_Preprint} that implies that two elliptic curves are geometrically isogenous if their reductions modulo $\mathfrak{p}$ are geometrically isogenous for all $\mathfrak{p}$ in a set of prime ideals of density $1$.

We do not know whether the conclusion of Theorem \ref{thm:modular_support_problem intro} can be strengthened to saying that $j = j'$ always. However, this strengthened conclusion is certainly false if we assume that $\mathfrak{p} \mid H_D(j) \Rightarrow \mathfrak{p} \mid H_D(j')$ holds just for infinitely many $D$ (and all $\mathfrak{p}$) instead of holding for all but finitely many $D$ (and all $\mathfrak{p}$). For instance, if $j$ and $j'$ are the two zeroes of $H_{-15}(T)$, we show in Theorem \ref{thm:the lonely theorem} that $\mathfrak{p} \mid H_D(j)$ if and only if $\mathfrak{p} \mid H_D(j')$ for all prime ideals $\mathfrak{p}$ in the ring of integers of $\mathbb{Q}(j) = \mathbb{Q}(j')$ and for all $D \in \mathbb{D}$ such that $D \equiv 1~\mathrm{mod}~8$.

\section{Preliminaries and Notation}

Throughout the paper, we adopt the following notation: if $K$ is a number field, we denote its ring of integers by $\mathcal{O}_K$ and for every finite set $S$ of maximal ideals of $\mathcal{O}_K$, we denote by $\mathcal{O}_{K,S}$ the ring of $S$-integers in $\mathcal{O}_K$. The norm $N(I)$ of a non-zero ideal $I \subseteq \mathcal{O}_{K,S}$ is the index $[\mathcal{O}_{K,S}:I]$. We have $N(IJ) = N(I)N(J)$ for all ideals $I,J \subseteq \mathcal{O}_{K,S}$. Finally, if $R$ is an arbitrary Dedekind domain and $a,b \in R$, we write $\gcd(a,b)$ for their greatest common divisor, \emph{i.e.}, for the ideal $Ra+Rb \subseteq R$.

\subsection{Elliptic curves with complex multiplication}

If $E$ is an elliptic curve over a field $K$, we denote its $j$-invariant by $j(E)$. If $L/K$ is a field extension, we denote the endomorphism ring of the base change $E_L$ of $E$ to $L$ by $\End_L(E)$. We say that $E$ has \emph{complex multiplication} if the canonical inclusion $\mathbb{Z} \hookrightarrow \End_{\overline{K}}(E)$ is strict for some algebraic closure $\overline{K}$ of $K$. We say that it has \emph{complex multiplication by a ring $\mathcal{O} \not\simeq \mathbb{Z}$} if $\End_{\overline{K}}(E) \simeq \mathcal{O}$.

Over a field of characteristic $0$, an elliptic curve with complex multiplication always has complex multiplication by an imaginary quadratic order $\mathcal{O}$. The $j$-invariants of the elliptic curves with complex multiplication by the imaginary quadratic order $\mathcal{O}$ of discriminant $D = \mathrm{disc}(\mathcal{O})$ are precisely the zeroes of the corresponding Hilbert class polynomial $H_D(T)$ ($D \in \mathbb{D}$). Recall that the Hilbert class polynomials $H_D(T)$ ($D \in \mathbb{D}$) all belong to $\mathbb{Z}[T]$ and are irreducible in this ring. The degree of $H_D$ will be denoted by $h(D)$ and equals the class number of the imaginary quadratic order of discriminant $D$. The discriminant of a singular modulus, \emph{i.e.}, of a zero of some $H_D(T)$, is the discriminant $D$ of the corresponding imaginary quadratic order.

Over fields of positive characteristic, the geometric endomorphism ring of an elliptic curve with complex multiplication is isomorphic either to an order in an imaginary quadratic field or to a maximal order in a quaternion algebra. In the first case or if the elliptic curve does not have complex multiplication, we call the elliptic curve \emph{ordinary}. In the second case, we call it \emph{supersingular}. For an ordinary elliptic curve over a finite field, all the geometric endomorphisms are already defined over the base field as we will now show.

\begin{lem} \label{lem:endomorphisms_ordinary_elliptic_curve}
Let $k$ be a finite field with an algebraic closure $\overline{k}$ and let $E_{/k}$ be an ordinary elliptic curve. Then $\End_k(E)=\End_{\overline{k}}(E)$, where we identify an endomorphism of $E$ with its base change to $\overline{k}$.
\end{lem}

\begin{proof}
Since $E$ is ordinary, the identity automorphism $\mathrm{id}_E$ and the Frobenius endomorphism $\pi$ are $\mathbb{Z}$-linearly independent by \cite[Chapter 13, Propositions 6.1 and 6.2]{Husemoeller_book}. Hence, $\End_k(E)$ has finite index $N$ inside $\End_{\overline{k}}(E)$. There exists a finite Galois extension $k \subseteq k'$ with $k' \subseteq \overline{k}$ such that $\End_{\overline{k}}(E) = \End_{k'}(E)$, where we identify an endomorphism of $E_{k'}$ with its base change to $\overline{k}$. By \cite[Theorem 14.84]{GoertzWedhorn}, $\End_k(E)$ is precisely the subset of $\End_{k'}(E)$ fixed by $\mathrm{Gal}(k'/k)$. Since $Nf \in \End_{k}(E)$ for all $f \in \End_{k'}(E)$ and the latter is an integral domain, it follows that all elements of $\End_{k'}(E)$ are fixed by $\mathrm{Gal}(k'/k)$ and so $\End_{\overline{k}}(E) = \End_{k'}(E) = \End_{k}(E)$ as desired.
\end{proof}

\subsection{Reduction of elliptic curves}\label{subsec:reduction_of_elliptic_curves}

When we say that an elliptic curve $E$ over a number field $K$ has \emph{(potential) good/bad reduction at a maximal ideal $\mathfrak{p}$} of $\mathcal{O}_{K,S}$, we mean that its base change to the completion $K_{\mathfrak{p}}$ of $K$ at $\mathfrak{p}$ has (potential) good/bad reduction in the sense of \cite[VII, Section 5]{Silverman_book_2009}. The elliptic curve $E$ has \emph{(potential) good ordinary/supersingular reduction at $\mathfrak{p}$} if its reduction at $\mathfrak{p}$ is (potentially) good and ordinary/supersingular. If $E$ does not have complex multiplication, Serre has proved that for $\mathfrak{p}$ in a set of maximal ideals of natural density $1$, the reduction of $E$ modulo $\mathfrak{p}$ is ordinary. This follows from \cite[Th\'eor\`eme 20 on p. 189 and Remarque 2 on p. 190]{Serre_1981} combined with the facts that the reduction of $E$ at a maximal ideal of prime norm $>3$ is supersingular if and only if the trace of the Frobenius endomorphism of the reduced elliptic curve is $0$ (see \cite[V, Exercise 5.10]{Silverman_book_2009}) and that the set of maximal ideals of prime norm has natural density $1$.

However, understanding the geometric endomorphism ring of the reductions of $E$ modulo $\mathfrak{p}$ for varying $\mathfrak{p}$ is a difficult problem. We will use the following theorem of Zarhin, which allows us to find infinitely many maximal ideals $\mathfrak{p} \subseteq \mathcal{O}_{K,S}$ at which $E$ has good ordinary reduction and such that the endomorphism ring of the reduction of $E$ at each such $\mathfrak{p}$ satisfies some prescribed local conditions.

\begin{thm}[Zarhin] \label{thm:Zarhin}
Let $L$ be an imaginary quadratic field and let $\mathcal{O} \subseteq L$ be an order. Let $K$ be a number field with ring of $S$-integers $\mathcal{O}_{K,S}$ for some fixed set of maximal ideals $S$ and consider an elliptic curve $E_{/K}$ without complex multiplication. Fix a non-empty finite set $\mathcal{P}$ of rational primes and set $\mathcal{O}_\ell:=\mathcal{O} \otimes_{\mathbb{Z}} \mathbb{Z}_\ell$ for all $\ell \in \mathcal{P}$. Define $\mathcal{A}$ to be the set of maximal ideals $\mathfrak{p} \subseteq \mathcal{O}_{K,S}$ such that
\begin{enumerate}
    \item the characteristic of the residue field $k_\mathfrak{p}$ at $\mathfrak{p}$ does not belong to $\mathcal{P}$,
    \item the curve $E$ has good ordinary reduction $E_\mathfrak{p}$ at $\mathfrak{p}$, and
    \item we have $\End_{k_\mathfrak{p}}(E_\mathfrak{p}) \otimes_\mathbb{Z} \mathbb{Z}_\ell \simeq \mathcal{O}_\ell$ for all $\ell \in \mathcal{P}$.
\end{enumerate}
Then $\mathcal{A}$ has positive density in the set of prime ideals of $\mathcal{O}_{K,S}$.
\end{thm}

\begin{proof}
If $S=\emptyset$, this is a special case of \cite[Theorem 1.3]{Zarhin_2017}, see \cite[Example 1.5]{Zarhin_2017} where the fact is needed that the set of ordinary primes for $E$ has density $1$ as remarked above. The theorem with $S\neq \emptyset$ follows easily from this case.
\end{proof}

If $E$ has complex multiplication, then the behaviour of the geometric endomorphism rings of its reductions modulo various maximal ideals $\mathfrak{p}$ of good reduction with residue characteristics $p$ is well-understood thanks to the work of Deuring \cite{Deuring}. It is connected to the value of the Kronecker symbol $\left(\frac{\cdot}{p}\right)$ at the discriminant of the geometric endomorphism algebra of $E$, see  \cite[Chapter 13, Theorem 12]{Lang_1987}. The following useful proposition tells us how the reduction behaviour of a Galois orbit of singular moduli at a fixed maximal ideal $\mathfrak{P}$ in the ring of integers of $\overline{\mathbb{Q}}$ changes if we replace their discriminant by its product with a power of the residue characteristic of $\mathfrak{P}$. This is certainly well-known to the expert, but we have not managed to find an appropriate reference in the literature, so we provide a proof here.

\begin{prop} \label{prop:Hilbert_class_polynomial_mod_p}
For every discriminant $D \in \mathbb{D}$, every prime $p \in \mathbb{N}$, and every $n \in \mathbb{N}$ we have
\[
H_{Dp^{2n}}(X)\equiv H_D(X)^{k} \mbox{ \emph{mod} } p, 
\]
where, if $\mathcal{O}$ denotes the order of discriminant $D$, we have
\[
k = \frac{h\left(Dp^{2n}\right)}{h(D)} = \frac{2p^{n-1}}{|\mathcal{O}^\ast|} \left(p-\left(\frac{D}{p}\right) \right).
\]
\end{prop}

\begin{proof}
By induction, it suffices to prove the statement only for $n=1$. We begin by showing that $H_{D}(X)$ and $H_{Dp^{2}}(X)$ have the same set of roots in an algebraic closure of $\mathbb{F}_p$. In order to do so, fix a prime $\mathfrak{P} \subseteq \overline{\mathbb{Q}}$ lying above $p$, and let $\mathcal{J}_D, \mathcal{J}_{Dp^2} \subseteq \overline{\mathbb{Q}}$ be the sets of singular moduli of discriminant $D$ and $Dp^2$ respectively. We claim that $\mathcal{J}_{D} \text{ mod } \mathfrak{P} = \mathcal{J}_{Dp^2} \text{ mod } \mathfrak{P}$.

Let $j \in \mathcal{J}_D$ and let $E_{/\overline{\mathbb{Q}}}$ be an elliptic curve with $j(E)=j$. Since $\mathcal{J}_D$ is a Galois orbit over $\mathbb{Q}$, we can fix embeddings $\overline{\mathbb{Q}} \hookrightarrow \mathbb{C}$ and $\mathcal{O} \hookrightarrow \mathbb{C}$ such that $E_{\mathbb{C}}$ is complex-analytically isomorphic to $\mathbb{C}/\mathcal{O}$. Consider a complex elliptic curve $E'$ whose analytification is isomorphic to $\mathbb{C}/(\mathbb{Z}+p\mathcal{O})$. The inclusion $\mathbb{Z}+p\mathcal{O} \subset \mathcal{O}$ induces an isogeny of degree $p$ from $E'$ onto $E_{\mathbb{C}}$ and $E'$ has complex multiplication by an order of discriminant $Dp^2$. It follows that $j(E') \in \overline{\mathbb{Q}}$ and so both $E'$ as well as the isogeny can be defined over $\overline{\mathbb{Q}}$. We will denote the corresponding elliptic curve over $\overline{\mathbb{Q}}$ by $E'$ as well. Letting $j':=j(E')$, one then has $\Phi_p(j,j')=\Phi_p(j',j)=0$ by \cite[Proposition 14.11]{Cox_book_2013}, where we recall that $\Phi_p$ denotes the $p$-th modular polynomial. Reducing this equality modulo $\mathfrak{P}$ and using Kronecker's congruence relation \cite[Theorem 11.18~(v)]{Cox_book_2013} we obtain
\begin{equation} \label{eq:modular_polynomial_mod_p}
    \Phi_p(j,j') \equiv (j-j'^p)(j^p-j') \equiv 0 \text{ mod } \mathfrak{P}
\end{equation}
so that either $j \equiv j'^p$ or $j' \equiv j^p \text{ mod } \mathfrak{P}$. Let now $K \subseteq \overline{\mathbb{Q}}$ be the Galois closure of $\mathbb{Q}(j,j')$ and set $\mathfrak{p}:=\mathfrak{P}\cap K$ with residue field $k_\mathfrak{p}$. Since the natural map from the decomposition group $D(\mathfrak{p}/p)$ of $\mathfrak{p}$ over $p$ to $\mathrm{Gal}(k_\mathfrak{p}/\mathbb{F}_p)$ is surjective by \cite[Chapter I, Proposition 9.4]{Neukirch_book}, there exists an element $\sigma \in D(\mathfrak{p}/p) \subseteq \mathrm{Gal}(K/\mathbb{Q})$ that reduces to the Frobenius automorphism of $k_\mathfrak{p}/\mathbb{F}_p$. It then follows from \eqref{eq:modular_polynomial_mod_p} that either $j \equiv \sigma(j')$ or $j \equiv \sigma^{-1}(j') \text{ mod } \mathfrak{p}$. Since singular moduli of the same discriminant form a full Galois orbit over $\mathbb{Q}$, we conclude that $j \text{ mod } \mathfrak{P} \in \mathcal{J}_{Dp^2} \text{ mod } \mathfrak{P}$ and hence that $\mathcal{J}_{D} \text{ mod } \mathfrak{P} \subseteq \mathcal{J}_{Dp^2} \text{ mod } \mathfrak{P}$. The other inclusion can be shown to hold true by repeating the same argument with $j, D, \mathcal{O}$ interchanged with $j', Dp^2, \mathbb{Z}+p\mathcal{O}$. This proves the claim.

Let now $L_D$, $L_{Dp^2} \subseteq \overline{\mathbb{Q}}$ be the ring class fields associated to the orders of discriminant $D$ and $Dp^2$ respectively. Then the extension $L_D \subseteq L_{Dp^2}$ is Galois and $[L_{Dp^2}:L_D]=h(Dp^2)/h(D) =: k$. Recall that
\[ h(Dp^2)/h(D) = \frac{2}{|\mathcal{O}^{\ast}|} \left(p-\left(\frac{D}{p}\right) \right)\]
thanks to \cite[Corollary 7.28]{Cox_book_2013}. Moreover, by \cite[Proposition 2.3~(1)]{LiLiOuyang_Preprint}, every prime of $L_D$ lying above $p$ is totally ramified in $L_D \subseteq L_{Dp^2}$. Set $\mathfrak{q}:=\mathfrak{P}\cap L_D$ and $\mathfrak{Q}:=\mathfrak{P}\cap L_{Dp^2}$ the unique prime of $L_{Dp^2}$ lying above it. Fix moreover $j\in L_D$ to be any singular modulus of discriminant $D$. Then by the above claim, there exists a singular modulus $j' \in L_{Dp^2}$ of discriminant $Dp^2$ such that $j' \equiv j \text{ mod } \mathfrak{Q}$. Since $\mathfrak{q}$ is totally ramified in $L_D \subseteq L_{Dp^2}$, we also have $\sigma(j') \equiv j \text{ mod } \mathfrak{Q}$ for every $\sigma \in \mathrm{Gal}(L_{Dp^2}/L_D)$. Moreover, $j'$ is a primitive element for $L_{Dp^2}$ over $L_D$, so it has exactly $k$ Galois conjugates over $L_D$. Hence, this argument shows that for every singular modulus $j$ of discriminant $D$, there exist at least $k$ distinct singular moduli of discriminant $Dp^2$ reducing to $j$ modulo $\mathfrak{P}$. On the other hand, there are exactly $h(Dp^2)=kh(D)$ singular moduli of discriminant $Dp^2$ and $h(D)$ singular moduli of discriminant $D$. Since $H_{Dp^2}$ and $H_D$ are both monic, we conclude that $H_{Dp^2} (X)  \equiv H_D(X)^k \text{ mod } p$ as we wanted to show.
\end{proof}

\subsection{Degrees in function fields}\label{subsec:deg}

Let $F$ be an algebraically closed field of arbitrary characteristic and let $R$ be the coordinate ring of a smooth affine irreducible curve $\mathcal{C}$ over $F$. By \cite[Examples 15.2~(2)]{GoertzWedhorn}, $R$ is a Dedekind domain. It is an example of the kind of Dedekind domains with which we will work in this article. Therefore, the following technical machinery will be useful for us later.

Let $K$ denote the fraction field of $R$. Any $f \in K$ has a degree $\deg f \in \mathbb{N} \cup \{0\}$, defined by $\deg f = [K:F(f)]$ if the field extension $F(f) \subseteq K$ is finite (equivalently: algebraic) and $\deg f = 0$ otherwise. In the latter case, $f$ must belong to $F$ since $F$ is algebraically closed and the transcendence degree is additive in towers of field extensions. If $F(f) \subseteq K$ is finite, we can apply \cite[Proposition 15.31]{GoertzWedhorn} with $C_2 = \mathbb{P}^1_F$, $D \in \{[0],[\infty]\}$, $C_1$ equal to a smooth projective irreducible curve with an open immersion $\mathcal{C} \hookrightarrow C_1$ (cf. \cite[Corollary 6.32, Remark 15.15~(3), and Theorem 15.21]{GoertzWedhorn}), and $f$ equal to the finite morphism $C_1 \to C_2$ induced by $f$ (that we also denote by $f$). It follows that
\begin{equation}\label{eq:sumofvaluations}
    \deg f = -\sum_{x \in C_1(F)}{\min\{0,\mathrm{ord}_x(f)\}} = \sum_{x \in C_1(F)}{\max\{0,\mathrm{ord}_x(f)\}},
\end{equation}
where $\mathrm{ord}_x(f)$ denotes the order of vanishing of $f$ at $x$. Note that \eqref{eq:sumofvaluations} also holds if $f \in F$. We deduce that $\deg f$ coincides with $H(f)$ as defined in \cite[p.~8, equation (2), and p.~96, after equation (1)]{Mason}.

For future reference, we summarize some basic properties of the degree in the following proposition:

\begin{prop}\label{prop:prop of deg}
The degree map has the following properties:
\begin{enumerate}
    \item $\deg(fg) \leq \deg f+\deg g$ for all $f,g \in K$,
    \item $\deg(f^n) = |n|\deg f$ for all $f \in K^{\ast}$ and all $n \in \mathbb{Z}$, and
    \item $\deg f = 0$ if and only if $f \in F$.
\end{enumerate}
\end{prop}

\begin{proof}
Property (1) as well as Property (2) are clear from \eqref{eq:sumofvaluations}. Property (3) has already been established above.
\end{proof}

Since $R$ is a Dedekind domain, we may also define the degree $\deg I$ of a non-zero ideal $I$ of $R$ by stipulating that
\begin{enumerate}
    \item $\deg R = 0$,
    \item $\deg I = 1$ if $I$ is a maximal ideal, and
    \item $\deg(IJ) = \deg I + \deg J$ for all non-zero ideals $I$ and $J$.
\end{enumerate}
Note that $\deg f$ and $\deg(fR)$ are not equal in general, \emph{e.g.}, if $f \in R^{\ast}\backslash F^{\ast}$.

\section{G.C.D.'s and Hilbert class polynomials}\label{sec:functionfield}

In this section, we study the ``size'' of greatest common divisors of the form $\gcd (H_{D}(a), H_{D}(b))$ for varying $D\in \mathbb{D}$ and fixed $a,b$ belonging to several Dedekind domains of interest. As explained in the introduction, this problem is a modular analogue of the more classical question concerning the size of $\gcd (a^n-1, b^n-1)$ for $n \in \mathbb{N}$. In this setting, one usually assumes $a$ and $b$ to be multiplicatively independent since otherwise the greatest common divisors involved become trivially large. Equivalently, one assumes that over the fraction field of the Dedekind domain under consideration, the point $(a,b)$ is not contained in any proper special subvariety of $\mathbb{G}^2_m$. Using the dictionary between the multiplicative and the modular world, one sees that this condition translates into taking $a$ and $b$ as the $j$-invariants of elliptic curves without complex multiplication that are not geometrically isogenous to each other.

We begin with the function field case in characteristic $0$, \textit{i.e.}, with the case where $a$ and $b$ are assumed to be elements of the coordinate ring $R$ of a smooth irreducible affine curve defined over an algebraically closed field $F$ of characteristic $0$.

\begin{thm} \label{thm: gcd in function fields}
Let $F$ be an algebraically closed field of characteristic $0$, let $R$ be the coordinate ring of a smooth affine irreducible curve $\mathcal{C}_{/F}$, and let $A,B \in R$. If $\Phi_N(A,B) \neq 0$ for all $N \in \mathbb{N}$, then there exists a non-zero ideal $J \subseteq R$ such that
\[ \gcd(H_{D_1}(A),H_{D_2}(B)) \mid J \quad \mbox{ or } \quad H_{D_1}(A)H_{D_2}(B) = 0  \]
for all $D_1, D_2 \in \mathbb{D}$.
\end{thm}

Note that the second alternative in Theorem \ref{thm: gcd in function fields} can only occur if either $A$ or $B$ is constant and equal to a singular modulus. Theorem \ref{thm: gcd in function fields} is the analogue of the first part of \cite[Theorem 1]{AilonRudnick} by Ailon and Rudnick (with $F = \mathbb{C}$ and $R = \mathbb{C}[X]$) that one obtains by substituting $H_{D_1}(T)$, $H_{D_2}(T)$ for $T^k-1$ and ``not $j$-invariants of geometrically isogenous elliptic curves" for ``multiplicatively independent". Our proof goes along the lines of the proof in \cite{AilonRudnick}, but we apply Andr\'e's theorem for $Y(1)^2$ instead of Ihara--Serre--Tate's theorem for $\mathbb{G}^2_m$.

\begin{proof}
Suppose first that $A \in F$. Then for all $D \in \mathbb{D}$, either $H_D(A) = 0$ or $H_D(A) \in R^{\ast}$ and so the theorem holds with $J = R$. The same argument works if $B \in F$.

From now on, we assume that $A \not\in F$ and $B \not\in F$. Set $J_{D_1,D_2} = \gcd(H_{D_1}(A),H_{D_2}(B))$ for $D_1,D_2 \in \mathbb{D}$ and suppose that some maximal ideal $\mathfrak{m}$ of $R$ divides $J_{D_1,D_2}$. We want to show that $\mathfrak{m}$ has to belong to a finite set that is independent of $D_1,D_2$.

The tuple $(A,B)$ defines a morphism $\varphi: \mathcal{C} \to Y(1)_F^2 \simeq \mathbb{A}^2_F$. Let $\mathcal{C}'$ denote the Zariski closure of the image of $\varphi$. Since $A$ is non-constant by assumption, $\mathcal{C}'$ is a curve and $\varphi$ has finite fibers by \cite[Theorems 5.22~(3) and 10.19 and Proposition 15.16~(1)]{GoertzWedhorn}. Now, the maximal ideal $\mathfrak{m}$ corresponds to a point $Q_{\mathfrak{m}} \in \mathcal{C}(F)$. Since $\mathfrak{m}$ divides $J_{D_1,D_2}$, we deduce that $P_{\mathfrak{m}} := \varphi(Q_{\mathfrak{m}}) \in \mathcal{C}'(F)$ is a special point. It follows from our assumptions and from Andr\'e's theorem \cite{Andre} that the number of special points lying on $\mathcal{C}'$ is finite (the proof over $\mathbb{C}$ yields a proof over $F$ by standard arguments). So $P_{\mathfrak{m}}$ indeed belongs to a finite set that is independent of $D_1,D_2$. Since $\varphi$ has finite fibers, the same holds for $Q_{\mathfrak{m}}$. Because the correspondence between $\mathfrak{m}$ and $Q_{\mathfrak{m}}$ is a bijection, the ideal $\mathfrak{m}$ lies in a finite set that does not depend on $D_1,D_2$.

It remains to show that the order with which a given maximal ideal $\mathfrak{m}$ divides $J_{D_1,D_2}$ is bounded independently of $D_1,D_2$. Set $e(\mathfrak{m})$ equal to the supremum of the orders with which $\mathfrak{m}$ divides $A-\sigma \neq 0$, where $\sigma$ runs over the set of all singular moduli. We have $e(\mathfrak{m}) < \infty$ since at most one $A-\sigma$ can be divisible by $\mathfrak{m}$. But $H_{D_1}(A)$ factors as a product of pairwise coprime elements $A-\sigma$, where $\sigma$ runs over the singular moduli of discriminant $D_1$. So the order to which $\mathfrak{m}$ divides $J_{D_1,D_2}$ is bounded by the order to which it divides $H_{D_1}(A)$, which is in turn bounded by $e(\mathfrak{m})$. The theorem now follows.

\end{proof}

\begin{rmk}
If we do not assume that $\Phi_N(A,B)\neq 0$ for all $N\in \mathbb{N}$, the theorem becomes false. More precisely, if $\Phi_N(A,B)= 0$ for some positive integer $N$, then for every $D\in \mathbb{D}$ and for every maximal ideal $\mathfrak{m}$ dividing $H_D(A)$, there exists $D' \in \mathbb{D}$ such that $\mathfrak{m} \mid H_{D'}(B)$. Indeed, with the notation from the proof of Theorem \ref{thm: gcd in function fields}, if $Q_\mathfrak{m}$ is the point of $\mathcal{C}(F)$ corresponding to $\mathfrak{m}$, then $A(Q_\mathfrak{m})$ is a singular modulus of discriminant $D$. Since by specialization $\Phi_N(A(Q_\mathfrak{m}),B(Q_\mathfrak{m}))= 0$, we deduce that any two elliptic curves over $F$ with $j$-invariants $A(Q_\mathfrak{m})$ and $B(Q_\mathfrak{m})$ are isogenous. It follows that also $B(Q_\mathfrak{m})$ is a singular modulus of some discriminant $D' \in \mathbb{D}$. This precisely means that $\mathfrak{m} \mid H_{D'}(B)$. Note finally that, if $A \not \in F$, the set of maximal ideals dividing some $H_D(A)$ is infinite. Indeed, the image of the morphism $\varphi_A: \mathcal{C} \to \mathbb{A}^1_F$ defined by $A \in R \backslash F$ is then open and dense, hence cofinite in $\mathbb{A}^1_F$ by \cite[Proposition 15.4~(1)]{GoertzWedhorn}. In particular, all but finitely many singular moduli are in the image of $\varphi_A$ and the result easily follows.
\end{rmk}

\begin{cor}\label{cor: gcd in function fields}
Let $F$ be an algebraically closed field of characteristic $0$, let $R$ be the coordinate ring of a smooth affine irreducible curve $\mathcal{C}_{/F}$, and let $A,B \in R\backslash F$. If $\Phi_N(A,B) \neq 0$ for all $N \in \mathbb{N}$, then for all but finitely many $(D_1,D_2) \in \mathbb{D}^2$ the elements $H_{D_1}(A)$ and $H_{D_2}(B)$ are coprime.
\end{cor}

\begin{proof}
Suppose by contradiction that $\gcd(H_{D_1}(A), H_{D_2}(B)) \neq R$ for infinitely many pairs of discriminants $(D_1,D_2) \in \mathbb{D}^2$. Then by Theorem \ref{thm: gcd in function fields}, there exists some maximal ideal $\mathfrak{m}$ of $R$ dividing $\gcd(H_{D_1}(A), H_{D_2}(B))$ for infinitely many $(D_1,D_2) \in \mathbb{D}^2$.

Let $Q_{\mathfrak{m}} \in \mathcal{C}(F)$ denote the point corresponding to $\mathfrak{m}$. If $\varphi: \mathcal{C} \to Y(1)_F^2$ denotes the morphism induced by $(A,B)$, then the coordinates of $\varphi(Q_{\mathfrak{m}}) \in Y(1)^2(F) \simeq F^2$ are singular moduli of discriminants $D_1,D_2$ respectively for infinitely many $(D_1,D_2) \in \mathbb{D}^2$. This is a contradiction and the corollary follows.
\end{proof}

In the case $F = \mathbb{C}$, $\mathcal{C} = \mathbb{A}^1_F$, and $R = F[X]$, it is easy to find polynomials $A,B \in R$ for which $H_{D_1}(A)$ and $H_{D_2}(B)$ are coprime for every choice of $(D_1,D_2) \in \mathbb{D}^2$. Indeed, it suffices to choose $A$ and $B$ in such a way that the specializations $A(\tau)$ and $B(\tau)$ at complex numbers $\tau \in \mathbb{C}$ are never both singular moduli. As an example, one could take $A=X$ and $B=X+a$, for $a \in \mathbb{C}$ not an algebraic integer. One can even find $A,B \in \mathbb{Z}[X]$ with this property: for instance, the polynomials $A=X$ and $B=X+1$ satisfy $\gcd(H_{D_1}(A), H_{D_2}(B))=1$ for all $(D_1,D_2) \in \mathbb{D}^2$. This follows from the fact that differences of singular moduli are never units in the ring of algebraic integers, see \cite[Corollary 1.3]{Li_2018}.

We now turn to the function field case in positive characteristic. If $F$ is an algebraic closure of a finite field, then the statement of Theorem \ref{thm: gcd in function fields} is false as the following theorem shows.

\begin{thm} \label{thm: gcd singular moduli positive characteristic}
Let $p \in \mathbb{N}$ be prime and fix an algebraic closure $F = \overline{\mathbb{F}}_p$ of $\mathbb{F}_p$. Let $R$ be the coordinate ring of a smooth affine irreducible curve $\mathcal{C}_{/F}$ and let $A, B \in R\backslash F$. Then
\[
\limsup_{D \in \mathbb{D},\hspace{2pt}|D| \to \infty} \frac{\deg \left(\gcd(H_D(A), H_D(B)) \right)}{\deg H_D} > 0.
\]
\end{thm}

Theorem \ref{thm: gcd singular moduli positive characteristic} for $R = F[X]$ can be considered a slightly weaker modular analogue of \cite[Theorem 4]{Silverman_2004b}. The hypothesis that $F = \overline{\mathbb{F}}_p$ rather than an arbitrary field of characteristic $p$ is essential as the following example shows.

\begin{example}
Suppose that $F$ is an algebraic closure of $\mathbb{F}_p(S)$, where $S$ is an independent variable, $\mathcal{C} = \mathbb{A}^1_F$, and $R = F[X]$. Then $H_{D_1}(X-S)$ and $H_{D_2}(X-S^2)$ are coprime for all $D_1, D_2 \in \mathbb{D}$.
\end{example}

In order to prove Theorem \ref{thm: gcd singular moduli positive characteristic}, we will make use of the following preliminary result.

\begin{prop}\label{prop:non-trivial gcd in positive characteristic}
Let $p \in \mathbb{N}$ be prime and let $F$, $\mathcal{C}$, $R$, $A$, and $B$ be as in Theorem \ref{thm: gcd singular moduli positive characteristic}. Let $\Lambda \subseteq \mathbb{F}_p[T]$ be any set of polynomials such that $\{t \in F; P(t) = 0 \mbox{ for some } P \in \Lambda\} = F$. Then there exist infinitely many $\alpha \in \mathcal{C}(F)$ for which there is a polynomial $P \in \Lambda$ satisfying $P(A(\alpha)) = P(B(\alpha)) = 0$.
\end{prop}

\begin{proof}
The idea is to find infinitely many $\alpha \in \mathcal{C}(F)$ such that $B(\alpha)$ is the image of $A(\alpha)$ under some $k$-th power of the Frobenius automorphism of $F$, where $k = k(\alpha) \in \mathbb{N} \cup \{0\}$ is allowed to depend on $\alpha$. Since by assumption there exists $P \in \Lambda$ such that $P(B(\alpha))=0$ and the polynomial $P$ has integer coefficients, one then has $0=P(B(\alpha))=P(A(\alpha)^{p^k})=P(A(\alpha))^{p^k}$ for all these $\alpha$ and the result follows.

In order to find infinitely many such points $\alpha$, we will use Mason's function field version \cite[Lemma 10 on p. 97]{Mason} of the $abc$ conjecture in characteristic $p$, which has been proven by Mason. To this aim, we let $m \in \mathbb{N} \cup \{0\}$ denote the biggest non-negative integer such that there exists $A_0 \in R$ with $A = A_0^{p^m}$. Note that $m$ is well-defined thanks to Proposition \ref{prop:prop of deg}~$(2)$-$(3)$ and the fact that $A$ is non-constant by assumption.

For every $n \in \mathbb{N}$, set $P_n:=A_0 - B^{p^n} \in R$. We have $P_n \neq 0$ thanks to the maximality of $m$, and obviously
\begin{equation} \label{eq:S-unit eqaution positive char}
    A_0-B^{p^n}-P_n=0.
\end{equation}
We want to apply \cite[Lemma 10 on p. 97]{Mason}  to the above equality. First of all, note that $A_0/(-B^{p^n})$ is not a $p$-th power in the fraction field of $R$ since $A_0$ is not a $p$-th power in $R$ and $R$ is integrally closed by \cite[Corollary 6.32, Lemma 6.38~(1), Corollary 6.39, and Proposition B.70~(1)]{GoertzWedhorn}. Moreover, by Proposition \ref{prop:prop of deg} we also have
\begin{equation} \label{eq:degree_goes_to_infinity}
    \deg(A_0/(-B^{p^n})) \geq p^n\deg(B)-\deg(A_0)
\end{equation}
and the right-hand side goes to infinity as $n$ increases since $B \not\in F$ by hypothesis. Let now $\overline{\mathcal{C}}$ be the smooth projective closure of the curve $\mathcal{C}$ (see \cite[Corollary 6.32, Remark 15.15~(3), and Theorem 15.21]{GoertzWedhorn}) and let $S \subseteq \overline{\mathcal{C}}(F)$ be an arbitrary finite set of points containing $\overline{\mathcal{C}}(F) \backslash \mathcal{C}(F)$ as well as all the zeroes of $A_0$ and $B$. Inequality \eqref{eq:degree_goes_to_infinity} shows that if we choose $n$ large enough, then the degree of $A_0/(-B^{p^n})$ will be strictly bigger than $|S| +2g-2$, where $g=g(\overline{\mathcal{C}})$ is the genus of $\overline{\mathcal{C}}$. Hence, by \cite[Lemma 10 on p. 97]{Mason}, there exists $\alpha \in \mathcal{C}(F)\backslash S$ such that the orders of vanishing at $\alpha$ of the functions $A_0$, $B^{p^n}$, and $P_n$ are not all equal; note that the set $\overline{\mathcal{C}}(F)$ is in bijection with the set of valuations on the fraction field of $R$ constructed in \cite[Chapter VI]{Mason}. Since our choice of $S$ implies that $A_0$ and $B$ both do not vanish at $\alpha$, we deduce that $P_n(\alpha)=0$. This means precisely that $A(\alpha) = B(\alpha)^{p^{m+n}}$. Now, by repeatedly enlarging the set $S$, one can find infinitely many points $\alpha \in \mathcal{C}(F)$ with this property. We have reached the desired conclusion.
\end{proof}

We remark that, by Deuring's lifting theorem \cite[Chapter 13, Theorem 14]{Lang_1987} and the fact that every elliptic curve over $\overline{\mathbb{F}}_p$ has endomorphism ring larger than $\mathbb{Z}$, the set $\Lambda = \{H_D(T) \textrm{ mod } p; D \in \mathbb{D}\}$ satisfies the hypothesis of Proposition \ref{prop:non-trivial gcd in positive characteristic}.

\begin{proof}[Proof of Theorem \ref{thm: gcd singular moduli positive characteristic}]

Our goal is to find a sequence $\{D_k\}_{k \in \mathbb{N}}$ of discriminants $D_k \in \mathbb{D}$ such that $|D_k| \to \infty$ as $k \to \infty$ and
\begin{equation} \label{eq: sequence of discriminants going to infinity}
    \liminf_{k \to \infty} \frac{\deg \left(\gcd(H_{D_k}(A), H_{D_k}(B)) \right)}{\deg H_{D_k}} > 0.
\end{equation}

Proposition \ref{prop:non-trivial gcd in positive characteristic} applied to $\Lambda = \{H_D(T) \textrm{ mod } p; D \in \mathbb{D}\}$ implies that there exists some discriminant $D_0 \in \mathbb{D}$ and some $\alpha \in \mathcal{C}(F)$ such that $H_{D_0}(A(\alpha)) = H_{D_0}(B(\alpha)) = 0$. It follows from the definition of the degree of an ideal that
\[ 
\deg(\gcd(H_{D_0}(A),H_{D_0}(B))) \geq 1.
\]

For $k \in \mathbb{N}$, set $D_k = D_0p^{2k}$. It is clear that $|D_k| \to \infty$ as $k \to \infty$. We can now apply Proposition \ref{prop:Hilbert_class_polynomial_mod_p} to deduce that 
\[ \deg(\gcd(H_{D_k}(A),H_{D_k}(B))) \geq \frac{\deg H_{D_k}}{\deg H_{D_0}}\deg(\gcd(H_{D_0}(A),H_{D_0}(B))) \geq \frac{\deg H_{D_k}}{\deg H_{D_0}} > 0.\]

So \eqref{eq: sequence of discriminants going to infinity} holds and the theorem follows.
\end{proof}

We finally leave the cozy realm of function fields to enter the more hostile world of number fields. Influenced by the previous discussion, one may be tempted to believe that the modular analogues of the multiplicative results of Bugeaud--Corvaja--Zannier \cite{Bugeaud_Corvaja_Zannier_2003} and Corvaja--Zannier \cite{Corvaja_Zannier_2005} should also hold true in this setting. For instance, the aforementioned results inspire the following natural conjecture in the modular framework: for all ``well-chosen'' $a,b \in \mathbb{Z}$ and for every $\varepsilon>0$ we have
\[
\log \left(\gcd (H_D(a), H_D(b) ) \right) < \varepsilon \deg H_D
\]
provided that $D \in \mathbb{D}$ is sufficiently large in absolute value. Here ``well-chosen'' means that neither $a$ nor $b$ is a singular modulus and that furthermore $a$ and $b$ are not $j$-invariants of geometrically isogenous elliptic curves over $\mathbb{Q}$, conditions that, as we have already remarked above, correspond to multiplicative independence in the multiplicative setting. Perhaps surprisingly, this just stated conjecture turns out to be completely false in general. The reason for this is the possible existence of common supersingular primes for the elliptic curves having $a$ and $b$ as their $j$-invariants.

\begin{thm} \label{thm: BCZ_fails}
Let $K$ be a number field and let $S$ be a finite set of maximal ideals of $\mathcal{O}_K$. Consider two elliptic curves ${E_1}_{/K}$, ${E_2}_{/K}$ with potential good reduction outside of $S$, i.e., such that $j(E_i) \in \mathcal{O}_{K,S}$ for $i = 1,2$. Suppose that there exists a prime ideal $\mathfrak{p} \subseteq \mathcal{O}_{K,S}$ at which both $E_1$ and $E_2$ have potential good supersingular reduction. Let $p$ denote the rational prime lying under $\mathfrak{p}$. Then

\[ \limsup_{D \in \mathbb{D},\hspace{2pt}|D| \to \infty}{(\deg H_D)^{-1}\log N(\gcd(H_D(j(E_1)),H_D(j(E_2))))} \geq \frac{\log p}{p-1} > 0.\]
\end{thm}

Here, the logarithm of the ideal norm $\log N$ plays the role of the degree $\deg$ of an ideal in the function field case. Both of them count with multiplicity the number of maximal ideals dividing a given non-zero element; $\deg$ weights each maximal ideal by $1$ while $\log N$ weights each maximal ideal by the logarithm of the cardinality of its residue field.

In order to prove Theorem \ref{thm: BCZ_fails}, we will crucially rely on an equidistribution result due to Michel \cite{Michel_2004} concerning supersingular reduction of CM elliptic curves. We briefly recall the content of this result and deduce some easy consequences.

Let $p \in \mathbb{N}$ be a rational prime and fix a prime $\mathfrak{p} \subseteq \overline{\mathbb{Q}}$ lying above it. Let $\overline{\mathbb{F}}_p$ denote the residue field of $\mathfrak{p}$, it is an algebraic closure of $\mathbb{F}_p$. Let $L$ be an imaginary quadratic field and assume that $p$ is inert in $L$. We denote by $\text{Ell}(\mathcal{O}_L)$ the set of $\overline{\mathbb{Q}}$-isomorphism classes of elliptic curves with complex multiplication by $\mathcal{O}_L$ and by $\text{Ell}_{\text{ss}}(\overline{\mathbb{F}}_p)$ the set of isomorphism classes of supersingular elliptic curves over $\overline{\mathbb{F}}_p$. Both sets are finite: indeed, the cardinality of $\text{Ell}(\mathcal{O}_L)$ equals the class number of $L$ while $\text{Ell}_{\text{ss}}(\overline{\mathbb{F}}_p)$ is finite by \cite[V, Theorem 3.1]{Silverman_book_2009}. Moreover, every class in $\text{Ell}(\mathcal{O}_L)$ can be represented by the base change $E_{\overline{\mathbb{Q}}}$ of $E$ to $\overline{\mathbb{Q}}$, where $E$ is an elliptic curve over a number field $L_0 \subseteq \overline{\mathbb{Q}}$ that has good reduction at $\mathfrak{p} \cap L_0$. We define $E_{\overline{\mathbb{Q}}} \text{ mod } \mathfrak{p}$ as the base change to $\overline{\mathbb{F}}_{p}$ of $E \text{ mod } (\mathfrak{p} \cap L_0)$. We then have a map
\begin{equation} \label{eq:supersingular_reduction_map}
    \Psi_{\mathfrak{p}, \mathcal{O}_L}: \text{Ell}(\mathcal{O}_L) \to \text{Ell}_{\text{ss}}(\overline{\mathbb{F}}_p), \hspace{0.5cm} [E_{\overline{\mathbb{Q}}}] \mapsto [E_{\overline{\mathbb{Q}}} \text{ mod } \mathfrak{p}]
\end{equation}
that is well-defined by \cite[II, Proposition 4.4]{Silverman_book_1994}.

\begin{thm}[Michel]\label{thm:supersingular_equidistribution}
Let $D$ denote the discriminant of $L$. There exist an absolute constant $\eta > 0$ as well as a constant $c = c(\mathfrak{p}) \in \mathbb{R}$ such that for each class $[\widetilde{E}] \in \text{\emph{Ell}}_{\text{\emph{ss}}}(\overline{\mathbb{F}}_p)$, we have
\[ |\{[E] \in \text{\emph{Ell}}(\mathcal{O}_L); \Psi_{\mathfrak{p}, \mathcal{O}_L}([E]) = [\widetilde{E}]\}| \geq ((p-1)^{-1} - cD^{-\eta})h(D).\]
\end{thm}

\begin{proof}
We want to apply \cite[Theorem 3]{Michel_2004} with $G = G_K$. Following \cite{Michel_2004}, we denote the cardinality of $\text{Ell}_{\text{ss}}(\overline{\mathbb{F}}_p)$ by $n$ and its members by $e_1,\hdots,e_n$. We also use the probability measure $\mu_p$ on $\text{Ell}_{\text{ss}}(\overline{\mathbb{F}}_p)$ as defined in \cite[top of p. 189]{Michel_2004}.

If $p > 5$, it follows from \cite[III, Theorem 10.1]{Silverman_book_2009} and the definition of $\mu_p$ that $\mu_p(e_i) \geq (3n)^{-1}$ for all $i = 1,\hdots,n$. Furthermore, we have \[(3n)^{-1} \geq 4/(p+13) \geq (p-1)^{-1}\]
thanks to \cite[Chapter 13, Table 1 on p. 264]{Husemoeller_book} and the fact that $p > 5$.

If $p \in \{2,3,5\}$, then $|\text{Ell}_{\text{ss}}(\overline{\mathbb{F}}_p)| = 1$ thanks to \cite[Chapter 13, Theorem 4.1 and the following paragraphs up to and including Table 1]{Husemoeller_book} so that $n = 1$ and $\mu_p(e_1) = 1 \geq (p-1)^{-1}$.

Hence, Theorem \ref{thm:supersingular_equidistribution} follows from \cite[Theorem 3]{Michel_2004} with $G = G_K$.
\end{proof}

We can now prove Theorem \ref{thm: BCZ_fails}.

\begin{proof}[Proof of Theorem \ref{thm: BCZ_fails}]
Let $\kappa \in \left(0,\frac{1}{p-1}\right)$ be arbitrary and fix an algebraic closure $\overline{K}$ of $K$ as well as a prime $\mathfrak{P} \subseteq \overline{K}$ lying over $\mathfrak{p}$. By Theorem \ref{thm:supersingular_equidistribution}, there exist fundamental discriminants $D \in \mathbb{D}$ with $|D|$ arbitrarily large such that $p$ is inert in $\mathbb{Q}(\sqrt{D}) \subseteq \overline{K}$ and
\begin{equation}\label{eq:theabove}
\left |\left \{j \in \overline{K}; H_D(j) = 0, j \equiv j(E_i)\mbox{ mod } \mathfrak{P} \right \} \right | \geq (\deg H_D)\kappa
\end{equation}
for $i = 1,2$. For any such $D$, let $K_D \subseteq \overline{K}$ denote the compositum of $K$ and the Hilbert class field of $\mathbb{Q}(\sqrt{D})$, and set $\mathfrak{P}_D = \mathfrak{P} \cap K_D$. By \cite[II, Theorem 4.1]{Silverman_book_1994}, $K_D$ is the splitting field of $H_D$ over $K(\sqrt{D})$. Together with inequality \eqref{eq:theabove}, this implies that there exist integers $e_{D,i} \geq (\deg H_D)\kappa$ such that
\[ H_D(j(E_i)) = \prod_{j \in K_D,\hspace{2pt}H_D(j)=0}{(j(E_i)-j)} \in \mathfrak{P}_D^{e_{D,i}}\]
for $i = 1,2$. Since $p$ does not ramify in $\mathbb{Q}(\sqrt{D})$, the extension $K \subseteq K_D$ is unramified over $\mathfrak{p}$ by \cite[II, Example 3.3]{Silverman_book_1994} and \cite[Propositions B.2.3 and B.2.4]{BombieriGubler}. We deduce that $H_D(j(E_i)) \in \mathfrak{p}^{e_{D,i}}$ for $i = 1,2$. Together with the lower bound for $e_{D,i}$, this implies that
\[ (\deg H_D)^{-1}\log N(\gcd(H_D(j(E_1)),H_D(j(E_2)))) \geq \kappa\log N(\mathfrak{p}) \geq \kappa \log p.\]
The theorem follows.
\end{proof}

\begin{rmk}
The attentive reader has certainly noticed the similarity between the statements of Theorems \ref{thm: gcd singular moduli positive characteristic} and \ref{thm: BCZ_fails}. However, arguing along the lines of the proof of Theorem \ref{thm: gcd singular moduli positive characteristic} does not suffice to prove Theorem \ref{thm: BCZ_fails}. First of all, we do not have an analogue of Proposition \ref{prop:non-trivial gcd in positive characteristic} that would allow us to find at least one discriminant whose associated greatest common divisor is non-trivial. Another obstacle is that passing from an order of discriminant $D$ to an order of discriminant $Dp^2$ gives rise to an extension of ring class fields that is totally ramified at each prime above $p$; for a proof of this fact, see, \emph{e.g.}, \cite[Proposition 2.3~(1)]{LiLiOuyang_Preprint}. Therefore, it is not clear whether the norm of the greatest common divisor increases at all when one passes from discriminant $D$ to discriminant $Dp^2$.

On the other hand, supersingular primes can be used to prove Theorem \ref{thm: gcd singular moduli positive characteristic} in some special cases: if there exists some $\alpha \in \mathcal{C}(F)$ such that both $A(\alpha)$ and $B(\alpha)$ are $j$-invariants of supersingular elliptic curves, then we can deduce Theorem \ref{thm: gcd singular moduli positive characteristic} from Theorem \ref{thm:supersingular_equidistribution} with a similar proof as the one of Theorem \ref{thm: BCZ_fails}. In this case, we can even strengthen Theorem \ref{thm: gcd singular moduli positive characteristic} to say that the limit superior is greater than or equal to $(p-1)^{-1}$. However, since the set of supersingular $j$-invariants in $F$ is finite, it is clear that, for some choices of $A,B \in R\backslash F$, no such $\alpha$ exists; for instance, take any $A \in R\backslash F$ and $B = A + b$, where $b \in F$ is not a difference of two supersingular $j$-invariants.
\end{rmk}

It is easy to construct examples where $E_1$ and $E_2$ have no complex multiplication and are not geometrically isogenous to each other, but nevertheless have a common prime of potential good supersingular reduction: fix any maximal ideal $\mathfrak{p}$ of $\mathcal{O}_{K,S}$. By \cite[Theorem 14.18]{Cox_book_2013} and since $0$ is supersingular in characteristics $2$ and $3$, we can choose (not necessarily distinct) supersingular $j$-invariants $j_{1,\mathfrak{p}},j_{2,\mathfrak{p}} \in \mathcal{O}_{K,S}/\mathfrak{p}$. All but finitely many of the lifts of $j_{1,\mathfrak{p}}$ to $\mathcal{O}_{K,S}$ are not singular moduli since the degree of a singular modulus equals the class number of the corresponding imaginary quadratic order and this goes to $\infty$ with the absolute value of the discriminant by \cite[Chapter XVI, Theorem 4]{Lang_book_1994} and \cite[Chapter I, Proposition 12.9]{Neukirch_book}. Fix any such lift $j_1 \in \mathcal{O}_{K,S}$. Then all but finitely many of the lifts of $j_{2,\mathfrak{p}}$ to $\mathcal{O}_{K,S}$ are $j$-invariants of elliptic curves without complex multiplication as above. Moreover, all but finitely many of these lifts are $j$-invariants of elliptic curves that are not geometrically isogenous to the elliptic curve with $j$-invariant $j_1$ since the existence of such an isogeny implies the existence of an isogeny whose degree is bounded in terms of $K$ and $j_1$ by the main theorem of \cite{Masser_Wuestholz_1990}. Hence, one can find many examples where the hypothesis of Theorem \ref{thm: BCZ_fails} holds although the two elliptic curves have no complex multiplication and are not geometrically isogenous to each other.

We know that there exist infinitely many common supersingular primes for $E_1$ and $E_2$ in the following cases:
\begin{enumerate}
    \item Both $E_1$ and $E_2$ have complex multiplication, see \cite[Chapter 13, Theorem 12]{Lang_1987}.
    \item Both $E_1$ and $E_2$ do not have complex multiplication, one of them can be defined over a number field with at least one real embedding, and $E_1$ and $E_2$ are geometrically isogenous, see \cite{Elkies_1989,Elkies_1987}.
\end{enumerate}

Based on \cite[Remark 2 on p. 37]{LangTrotter}, Fouvry and Murty conjecture in \cite[Equation (1.4)]{Fouvry_Murty_1995} that there are infinitely many common supersingular primes if both $E_1$ and $E_2$ do not have complex multiplication, are defined over $\mathbb{Q}$, and are not geometrically isogenous. In the same article, they also prove an averaged version of this conjecture. A similar averaged result is known if both $E_1$ and $E_2$ are defined over the rationals, $E_1$ has complex multiplication, and $E_2$ does not, see \cite[Theorem 10]{James_Smith_2013} and \cite[pp. 199-200]{DavidPappalardi}.

We can also consider the following modular version of \cite[Conjecture A]{AilonRudnick} by Ailon and Rudnick and \cite[Conjecture 10]{Silverman_2005} by Silverman: for which $a$, $b \in \mathcal{O}_{K,S}$ do there exist infinitely many $D \in \mathbb{D}$ such that $H_D(a)$ and $H_D(b)$ are coprime? Does it suffice to assume that $a \neq b$ or at least that neither $a$ nor $b$ is a singular modulus and $\Phi_N(a,b) \neq 0$ for all $N \in \mathbb{N}$? This problem would be trivial if $H_D(a) \in \mathcal{O}_{K,S}^{\ast}$ or $H_D(b) \in \mathcal{O}_{K,S}^{\ast}$ for infinitely many $D \in \mathbb{D}$. If $a$ and $b$ are not both singular moduli and $S \neq \emptyset$, then it is an open problem whether this can happen or not. We can at least show that $H_D(a)$ cannot belong to $\mathcal{O}_{K,S}^{\ast}$ for all but finitely many $D \in \mathbb{D}$:

\begin{thm}\label{thm:limsupprime}
Let $K$ be a number field and let $S$ be a finite set of maximal ideals of $\mathcal{O}_K$. Let $j \in \mathcal{O}_{K,S}$ and let $E_{/K}$ be an elliptic curve with $j(E)=j$. Then
\[
\limsup_{D \in \mathbb{D},\hspace{2pt}|D| \to \infty} P_{\emph{o}}(H_D(j)) = \infty,
\]
where $P_{\emph{o}}(a)$ is defined as follows: $P_{\emph{o}}(0) = \infty$ and for $a \in \mathcal{O}_{K,S}\backslash\{0\}$, $P_{\emph{o}}(a)$ denotes the largest norm of a prime factor of $a$ which is of good ordinary reduction for $E$ and $P_{\emph{o}}(a) = 1$ if no such prime factor exists.
\end{thm}

Note that, if $P(a)$ denotes the largest norm of a prime factor of $a \in \mathcal{O}_{K,S}\backslash(\mathcal{O}_{K,S}^{\ast} \cup \{0\})$ and $P(a) = \infty$ or $1$ for $a = 0$ and $a \in \mathcal{O}_{K,S}^{\ast}$ respectively, then $\lim_{D \in \mathbb{D},\hspace{2pt}|D| \to \infty}{P(H_D(j))} = \infty$ is equivalent to the fact that for every finite set $\widetilde{S}$ of maximal ideals of $\mathcal{O}_K$, there are at most finitely many $D \in \mathbb{D}$ such that $H_D(j)$ is an $\widetilde{S}$-unit. As mentioned above, it is still an open problem whether this is true or not. It is known if either $\widetilde{S} = \emptyset$ \cite{Habegger_2015} or if $j$ is a singular modulus \cite{HMRL_Preprint}. See \cite{BHK_Preprint, Cai, Campagna_2021, Campagna_Preprint, Stefan_Preprint} for work on making these results effective.

This actually provides another example where the analogy between Shimura varieties and algebraic groups fails to hold perfectly. Namely, there is a discrepancy between Habegger's \cite[Theorem 2]{Habegger_2015} in the modular case and \cite[Theorems 0.1 and 0.2]{BIR_2008} by Baker, Ih, and Rumely in the multiplicative and elliptic case: in the modular case, the fixed non-cuspidal point(s) can be completely arbitrary while in the multiplicative and in the elliptic case, the fixed point cannot be a torsion point. We can ask whether removing the supersingular prime factors eliminates this discrepancy as well: certainly, if $j \in \mathcal{O}_{K,S}$ is a singular modulus, then $N(H_D(j))$ is divisible only by supersingular primes for $j$ for infinitely many $D \in \mathbb{D}$. Does the converse hold as well?

\begin{proof}[Proof of Theorem \ref{thm:limsupprime}]
Fix an algebraic closure $\overline{K}$ of $K$. In the proof, we will repeatedly use the fact that there exist infinitely many primes of good ordinary reduction for $E$, see Section \ref{subsec:reduction_of_elliptic_curves}. The goal of our proof is to construct a monotonically decreasing sequence of discriminants $D$ for which $P_{\mathrm{o}}(H_D(j))$ goes to $\infty$. We do this recursively as follows. Let $\mathfrak{p}_1 \subseteq \mathcal{O}_{K,S}$ be a maximal ideal of good ordinary reduction for $E$ and choose a prime $\mathfrak{P}_1 \subseteq \overline{K}$ lying above it. By Deuring's Lifting Theorem \cite[Chapter 13, Theorem 14]{Lang_1987}, the fact that every elliptic curve over a finite field has complex multiplication, and Proposition \ref{prop:Hilbert_class_polynomial_mod_p}, there exist a discriminant $D_{\mathfrak{p}_1} \in \mathbb{Z}_{<0}$ and an elliptic curve ${E_{\mathfrak{p}_1}}_{/\overline{K}}$ with complex multiplication by the imaginary quadratic order of discriminant $D_{\mathfrak{p}_1}$ such that $E_{\mathfrak{p}_1} \text{ mod } \mathfrak{P}_1 \simeq E_{\overline{K}} \text{ mod } \mathfrak{P}_1$ and $H_{D_{\mathfrak{p}_1}}(j) \neq 0$. Note that this implies in particular that $\mathfrak{p}_1$ divides $H_{D_{\mathfrak{p}_1}}(j)$ so that $N(\mathfrak{p}_1) \leq P_{\mathrm{o}}(H_{D_{\mathfrak{p}_1}}(j)) < \infty$. 

Suppose now that we have constructed a sequence of primes $\mathfrak{p}_1$, \dots, $\mathfrak{p}_{n-1}$ such that $N(\mathfrak{p}_1) < \cdots < N(\mathfrak{p}_{n-1})$ and a sequence of discriminants $D_{\mathfrak{p}_1} > \cdots > D_{\mathfrak{p}_{n-1}}$ such that $N(\mathfrak{p}_m) \leq P_{\mathrm{o}}(H_{D_{\mathfrak{p}_m}}(j)) < \infty$ for all $m \in \{1,...,n-1\}$. We take some maximal ideal $\mathfrak{p}_n$ of good ordinary reduction for $E$ such that $N(\mathfrak{p}_n) > \max_{m=1,\hdots,n-1}\{P_{\mathrm{o}}(H_{D_{\mathfrak{p}_m}}(j))\}$.

For the same reasons as above, there exist a prime $\mathfrak{P}_n \subseteq \overline{K}$ lying above $\mathfrak{p}_n$, a discriminant $D_{\mathfrak{p}_n}$, and an elliptic curve ${E_{\mathfrak{p}_n}}_{/\overline{K}}$ with complex multiplication by the imaginary quadratic order of discriminant $D_{\mathfrak{p}_n}$ satisfying $E_{\mathfrak{p}_n} \text{ mod } \mathfrak{P}_n \simeq E \text{ mod } \mathfrak{P}_n$, $H_{D_{\mathfrak{p}_n}}(j) \neq 0$, and $|D_{\mathfrak{p}_n}| > |D_{\mathfrak{p}_{n-1}}|$. This implies that $N(\mathfrak{p}_n) \leq P_{\mathrm{o}}(H_{D_{\mathfrak{p}_n}}(j)) < \infty$.

Iterating this construction, we find a sequence of discriminants $(D_{\mathfrak{p}_n})_{n \in \mathbb{N}}$ with the desired properties. This concludes the proof.
\end{proof}

\begin{rmk}
For $a \in \mathcal{O}_{K,S}\backslash\{0\}$, let us denote by $P_{\mathrm{ss}}(a)$ the largest norm of a prime factor of $a$ which is of good supersingular reduction for $E$ and set $P_{\mathrm{ss}}(a) = 1$ if there is no such prime factor and $P_{\mathrm{ss}}(0) = \infty$. If there are infinitely many primes of good supersingular reduction for $E$, which is known in certain cases thanks to \cite{Elkies_1989,Elkies_1987}, then we can follow the proof of Theorem \ref{thm:limsupprime} to show that $\limsup_{D \in \mathbb{D},\hspace{2pt}|D| \to \infty}{P_{\mathrm{ss}}(H_D(j))} = \infty$.
\end{rmk}

\section{The support problem}\label{sec:supportproblem}

In the previous section, we studied $\gcd(H_D(a),H_D(b))$ for varying $D\in \mathbb{D}$ and fixed $a,b$ in some Dedekind domain. This greatest common divisor is as big as possible if $H_D(a)$ divides $H_D(b)$ (or \emph{vice versa}). We now want to investigate in which cases this can happen for all but finitely many discriminants $D$. In fact, we want to understand more generally in which cases the set of prime ideals dividing $H_D(a)$ (its \emph{support}) is contained in the set of prime ideals dividing $H_D(b)$ for all but finitely many $D$. If we pass from the modular to the multiplicative world and replace the polynomials $H_D(T)$ ($D \in \mathbb{D}$) by the polynomials $T^n-1$ ($n \in \mathbb{N}$), then this question has been answered by Corrales-Rodrig\'{a}\~{n}ez and Schoof \cite{Corrales-Rodriganez_Schoof_1997}.

We are therefore led to introduce a general setting which encompasses both versions of this problem: let $R$ be a Dedekind domain which is not a field and let $\mathcal{N}$ be an arbitrary countably infinite set. We are given a polynomial $f_n(T) \in R[T]$ for each $n \in \mathcal{N}$ and two elements $a,b \in R$. If for all but finitely many $n \in \mathcal{N}$, every prime ideal factor of $f_n(a)$ also divides $f_n(b)$, i.e., if the \emph{support property} holds for all but finitely many $n \in \mathcal{N}$, what can we say about $a$ and $b$? Following \cite{Corrales-Rodriganez_Schoof_1997}, we will refer to this question as the \textit{support problem for the polynomials $f_n(T) \in R[T]$} ($n \in \mathcal{N}$). Clearly, the answer cannot be of universal nature and it depends very much on $R$ and the polynomials $f_n(T)$ for $n \in \mathcal{N}$.

In many of the instances of the support problem that we will consider later, there are some trivial subcases that are easily dealt with. We present them in the following examples:

\begin{example}[Isotriviality] \label{ex:support_1}
Let $F$ be an algebraically closed field and let $R$ be the coordinate ring of a smooth affine irreducible curve $\mathcal{C}_{/F}$. Recall that $R$ is a Dedekind domain by \cite[Examples 15.2~(2)]{GoertzWedhorn}. In this example, we study the case where $f_n(T) \in F[T] \subseteq R[T]$ for all $n \in \mathcal{N}$ and either $a \in F$ or $b \in F$. We will see that the support problem is often trivial in this case.

\begin{enumerate}
    \item Set
\[
Z:=\{\tau \in F; f_n(\tau)=0 \mbox{ for infinitely many $n \in \mathcal{N}$ such that $f_n \neq 0$}\}.
\]
For instance, if the polynomials $f_n(T)$ ($n \in \mathcal{N}$) are pairwise coprime, then $Z = \emptyset$. For every $a \in F \backslash Z$ and for all but finitely many $n \in \mathcal{N}$ such that $f_n \neq 0$, the values $f_n(a)$ are non-zero and hence units in $R$. In particular, for every $b \in R$ we have that $f_n(a)$ divides $f_n(b)$ for all but finitely many $n \in \mathcal{N}$.

On the other hand, if $a \in Z$ and for all but finitely many $n \in \mathcal{N}$, every prime ideal dividing $f_n(a)$ also divides $f_n(b)$, then we deduce that $f_{n_0}(b) = 0$ for some $n_0 \in \mathcal{N}$ such that $f_{n_0} \neq 0$. It follows that $b \in F$ since $F$ is algebraically closed.

\item Let now $f_n(T) \in F[T]$ ($n \in \mathcal{N}$) and $b \in F$ be such that 
\[
Z:=\bigcup_{n \in \mathcal{N},f_n(b) \neq 0} \{\tau \in F; f_n(\tau)=0\}
\]
is infinite. For instance, if the polynomials $f_n(T)$ ($n \in \mathcal{N}$) are all non-constant and pairwise coprime, then this holds for every choice of $b \in F$. It in particular implies that the set of polynomials $f_n(T)$ that do not have $b$ as a zero is infinite. Since $f_n(T) \in F[T]$ and $b \in F$, the value $f_n(b)$ is a unit in $R$ as soon as $f_n(b) \neq 0$. If for all but finitely many $n \in \mathcal{N}$, every prime ideal factor of $f_n(a)$ also divides $f_n(b)$, it follows that $f_n(a)$ is a unit for all but finitely many $n \in \mathcal{N}$ such that $f_n(b) \neq 0$, which implies that $a-\tau \in R^{\ast}$ for all but finitely many $\tau \in Z$. Since $Z$ is infinite, the morphism $\mathcal{C} \to \mathbb{A}^1_F$ induced by $a$ must then be constant by \cite[Proposition 15.4~(1)]{GoertzWedhorn} and it follows that $a \in F$.

On the other hand, if $Z$ is finite and no element of $Z$ belongs to the image of the morphism $\mathcal{C} \to \mathbb{A}^1_F$ induced by $a$, then for all $n \in \mathcal{N}$, every prime ideal dividing $f_n(a)$ also divides $f_n(b)$ since either $f_n(b) = 0$ or $f_n(a)$ is equal, up to scaling by a constant in $F^{\ast}$, to a product of factors $a-\tau \in R^{\ast}$ for $\tau \in Z$.
\end{enumerate}
\end{example}

\begin{example}[Frobenius]\label{ex:support_2}
Suppose that $\mathbb{F}_q \subseteq R$ for some prime power $q$ and that $f_n(T) \in \mathbb{F}_q[T] \subseteq R[T]$ for all $n \in \mathcal{N}$. Then for every $c \in R$ and for all $k,\ell \in \mathbb{Z}$, $k,\ell \geq 0$, we have that every prime ideal that divides $f_n(c^{q^k}) = f_n(c)^{q^k}$ also divides $f_n(c^{q^{\ell}}) = f_n(c)^{q^{\ell}}$ for all $n \in \mathcal{N}$. Thus, if $a = c^{q^k}$ and $b = c^{q^{\ell}}$, then for all $n \in \mathcal{N}$, every prime ideal dividing $f_n(a)$ also divides $f_n(b)$.
\end{example}

In the case of the \emph{multiplicative support problem}, \emph{i.e.}, if $\mathcal{N} = \mathbb{N}$ and $f_n(T) := T^n-1$ for $n \in \mathcal{N}$, we have that $f_n(a) \mid f_n(a^k)$ for every $a \in R$, every $k \in \mathbb{N} \cup \{0\}$, and every $n \in \mathcal{N}$. If $a \in R^{\ast}$, this also holds for negative exponents $k$ since $f_n(a^{-1}) = -f_n(a)a^{-n}$. Answering a question of Erd\H{o}s at the 1988 number theory conference in Banff, Corrales-Rodrig\'{a}\~{n}ez and Schoof solved the multiplicative support problem in the number field case by proving that these are the only possibilities if $ab \neq 0$:

\begin{thm}[Corrales-Rodrig\'{a}\~{n}ez--Schoof, \cite{Corrales-Rodriganez_Schoof_1997}]\label{thm:multiplicative_support_1}
Let $K$ be a number field and let $S$ be a finite set of maximal ideals of $\mathcal{O}_K$. If $a,b \in \mathcal{O}_{K,S}\backslash\{0\}$ satisfy that for all but finitely many $n \in \mathbb{N}$, every prime ideal of $\mathcal{O}_{K,S}$ dividing $a^n-1$ also divides $b^n-1$, then $b = a^k$ for some $k \in \mathbb{Z}$.
\end{thm}

In the proof of Theorem \ref{thm:multiplicative_support_1} as well as in many of the proofs that follow, we will sometimes enlarge $S$ to a bigger set $S'$. When doing this, we will always identify the primes of $\mathcal{O}_{K,S'}$ as well as the primes of $\mathcal{O}_{K,S}$ with subsets of the primes of $\mathcal{O}_{K}$ by identifying $\mathfrak{p}$ with $\mathfrak{p} \cap \mathcal{O}_{K}$. We will also implicitly use that divisibility by primes of $\mathcal{O}_{K,S}$ outside of $S'$ is preserved when we replace $S$ by $S'$.

\begin{proof}
If the hypothesis holds for all $n \in \mathbb{N}$, then this is precisely \cite[Theorem 1]{Corrales-Rodriganez_Schoof_1997}. For the general case, let us assume that the hypothesis holds for $n > N_0$ for some $N_0 \in \mathbb{N}$ and let us see how the theorem follows from the previous case: if $a$ is not a root of unity, then one just has to replace $S$ by a set $S' \supset S$ that contains all prime factors of $\prod_{i=1}^{N_0}{(a^i-1)}$. If, on the other hand, $a$ is a root of unity, then $a^n-1 = 0$ for infinitely many $n \in \mathbb{N}$ and one deduces that $b$ is a root of unity and even, by taking $n$ equal to two consecutive sufficiently large multiples of the order of $a$, that the order of $b$ divides the order of $a$. Thus, $b$ is a power of $a$ as desired.
\end{proof}

The zeroes of $T^n-1$ in $\mathbb{C}$ are the roots of unity of order dividing $n$. It might seem more natural to look at the polynomials whose zeroes are the roots of unity of order precisely $n$ since these are the minimal polynomials over $\mathbb{Q}$ of the special points of $\mathbb{G}_{m,\mathbb{C}}$ just as the Hilbert class polynomials are the minimal polynomials over $\mathbb{Q}$ of the special points of $Y(1)_{\mathbb{C}}$. Thus, we replace the polynomials $T^n-1$ by the cyclotomic polynomials $\Psi_n(T)$ ($n \in \mathbb{N} =: \mathcal{N}$) to create the \emph{cyclotomic support problem}. To prove the analogue of Theorem \ref{thm:multiplicative_support_1}, we need the following lemma.

\begin{lem}\label{lem:order_modulo_prime}
Let $K$ be a number field, let $S$ be a finite set of maximal ideals of $\mathcal{O}_K$, and fix a maximal ideal $\mathfrak{p} \subseteq \mathcal{O}_{K,S}$ of residue characteristic $p$. The following are equivalent for $a \in \mathcal{O}_{K,S}\backslash \mathfrak{p}$ and $k \in \mathbb{N}$ coprime to $p$:
\begin{enumerate}
    \item $a$ has order $k$ modulo $\mathfrak{p}$,
    \item $\Psi_{kp^{\ell}}(a) \in \mathfrak{p}$ for some $\ell \in \mathbb{Z}_{\geq 0}$, and
    \item $\Psi_{kp^{\ell}}(a) \in \mathfrak{p}$ for all $\ell \in \mathbb{Z}_{\geq 0}$.
\end{enumerate}
\end{lem}

\begin{proof}
It is clear that (1) implies (2) and that (3) implies (1), just take $\ell=0$ and, for the second implication, use that the polynomial $T^k-1$ modulo $\mathfrak{p}$ is separable. Finally, for every $\ell \in \mathbb{Z}_{\geq 1}$, we have the identity
\[ \Psi_{kp^{\ell}}(T) \equiv (\Psi_k(T))^{(p-1)p^{\ell-1}} \text{ mod } p,\]
which one can deduce from the identity
\[ T^{kp^{\ell}} - 1 = (T^k-1)^{p^{\ell}} \text{ mod } p\] using double induction on $(k,\ell)$, and so (2) implies (3).
\end{proof}

We are now ready to solve the cyclotomic support problem as a direct consequence of Theorem \ref{thm:multiplicative_support_1}.

\begin{thm}\label{thm: cyclotomic_support_problem}
Let $K$ be a number field, let $S$ be a finite set of maximal ideals of $\mathcal{O}_K$, and fix $N_0 \in \mathbb{N}$. If $a,b \in \mathcal{O}_{K,S}\backslash\{0\}$ satisfy that for all $n \in \mathbb{N}$ with $n > N_0$, every prime ideal of $\mathcal{O}_{K,S}$ dividing $\Psi_n(a)$ also divides $\Psi_n(b)$, then either $a$ and $b$ are roots of unity of the same order or $b = a^{\pm 1}$.
\end{thm}

The conclusion of Theorem \ref{thm: cyclotomic_support_problem} is best possible: if $a$ and $b$ are roots of unity of the same order $k$ and $\mathfrak{p}$ is a maximal ideal of $\mathcal{O}_{K,S}$ of residue characteristic $p$, then Lemma \ref{lem:order_modulo_prime} implies that $a$ and $b$ both have order $k_0$ modulo $\mathfrak{p}$ where $p$ does not divide $k_0$ and $k = k_0p^{\ell_0}$ for some $\ell_0 \in \mathbb{Z}_{\geq 0}$. Hence, Lemma \ref{lem:order_modulo_prime} implies that for $n \in \mathbb{N}$, $\Psi_n(a) \in \mathfrak{p}$ if and only if $n = k_0p^{\ell}$ for some $\ell \in \mathbb{Z}_{\geq 0}$ if and only if $\Psi_n(b) \in \mathfrak{p}$. So the hypothesis of the support problem holds. The same is true if $b = a^{-1} \in \mathcal{O}^{\ast}_{K,S}$ since $a$ and $a^{-1}$ have the same order modulo every maximal ideal of $\mathcal{O}_{K,S}$.

\begin{proof}
If $a$ is a root of unity, $k$ is its order, and $\mathfrak{p}$ is a maximal ideal of $\mathcal{O}_{K,S}$ of residue characteristic $p$, then $\Psi_{kp^{\ell}}(a) \in \mathfrak{p}$ for all $\ell \in \mathbb{Z}_{\geq 0}$ by Lemma \ref{lem:order_modulo_prime}. Choosing $\ell$ large enough and applying our hypothesis on $a$ and $b$, we deduce that $\Psi_{kp^{\ell}}(b) \in \mathfrak{p}$ for some $\ell \in \mathbb{Z}_{\geq 0}$ and so, again by Lemma \ref{lem:order_modulo_prime}, $\Psi_k(b) \in \mathfrak{p}$. Varying $\mathfrak{p}$, we deduce that $\Psi_k(b) = 0$ and we are done.

If $a$ is not a root of unity, then, after enlarging $S$ if necessary, we can assume without loss of generality that $a \in \mathcal{O}_{K,S}^{\ast}$ and that  $\prod_{i=1}^{N_0}{\Psi_i(a)} \in \mathcal{O}^{\ast}_{K,S}$ so that the support property holds for all $n \in \mathbb{N}$. Observe now that such property also holds in the other direction: if $\mathfrak{p}$ is a prime ideal of $\mathcal{O}_{K,S}$ dividing $\Psi_n(b)$ for some $n \in \mathbb{N}$, then $\mathfrak{p}$ also divides $\Psi_n(a)$. Indeed, let $p$ be the residue characteristic of $\mathfrak{p}$ and let $m$ be the order of $a$ modulo $\mathfrak{p}$, which is well-defined since $a \not\in \mathfrak{p}$. It follows that $\mathfrak{p} \mid \Psi_m(a)$ and the support property implies that $\mathfrak{p} \mid \Psi_m(b)$. By Lemma \ref{lem:order_modulo_prime}, $b$ has order $m$ modulo $\mathfrak{p}$. On the other hand, since we assumed that $\mathfrak{p} \mid \Psi_n(b)$, writing $n=n_0 p^\ell$ with $p\nmid n_0$, we must have $m=n_0$ by Lemma \ref{lem:order_modulo_prime}. Using once more Lemma \ref{lem:order_modulo_prime}, we deduce that $\mathfrak{p} \mid \Psi_{n_0 p^\ell}(a)=\Psi_{n}(a)$, as wanted.

We conclude that for all $n \in \mathbb{N}$, a prime ideal of $\mathcal{O}_{K,S}$ divides $a^n-1 = \prod_{i \mid n}{\Psi_i(a)}$ if and only if it divides $\prod_{i \mid n}{\Psi_i(b)} = b^n-1$. Thus, Theorem \ref{thm:multiplicative_support_1} implies that $b = a^k$ and $a=b^r$ for some $k,r \in \mathbb{Z}$. Hence we have $a^{|kr-1|}=1$ and, since by assumption $a$ is not a root of unity, this yields $kr = 1$, so $k \in \{ \pm 1 \}$. The proof is concluded.
\end{proof}

The next theorem resolves the function field case of the multiplicative and the cyclotomic support problem in characteristic $0$.

\begin{thm}\label{thm:g_m_multiplicative_support_problem_function_field_char_0}
Let $R$ be the coordinate ring of a smooth affine irreducible curve $\mathcal{C}$ over an algebraically closed field $F$ of characteristic $0$. Let $A, B \in R \backslash F$ and let $N_0 \in \mathbb{N}$. The following hold:
\begin{enumerate}
    \item Suppose that for all $n \in \mathbb{N}$ with $n > N_0$, every prime ideal of $R$ that divides $A^n-1$ also divides $B^n-1$. Then $B = A^k$ for some $k \in \mathbb{Z}\backslash\{0\}$.
    \item Suppose that for all $n \in \mathbb{N}$ with $n > N_0$, every prime ideal of $R$ that divides $\Psi_n(A)$ also divides $\Psi_n(B)$. Then $B = A^{\pm 1}$.
\end{enumerate}
\end{thm}

The case where $A \in F$ or $B \in F$ is uninteresting, see Example \ref{ex:support_1}. Note that the set $Z$ in Example \ref{ex:support_1}~(2) is in both cases infinite for all $B \in F\backslash\{1\}$ (and in case (2) even for all $B \in F$).

\begin{proof}
Throughout the proof, we treat both cases simultaneously until the very end, where case $(2)$ receives some extra attention. The tuple $(A,B)$ defines a rational map $\varphi: \mathcal{C} \dashrightarrow \mathbb{G}_{m,F}^2$. Let $\mathcal{C}'$ denote the Zariski closure of the image of $\varphi$. Since $A$ is non-constant, $\mathcal{C}'$ is a curve and $\varphi$ has finite fibers. Thanks to \cite[Proposition 15.4~(1)]{GoertzWedhorn}, applied to the morphism $\mathcal{C} \to \mathbb{A}^1_F$ induced by $A$, we can assume, after increasing $N_0$ if necessary, that for every root of unity $\zeta$ of order $n > N_0$ there exists some maximal ideal $\mathfrak{m}_{\zeta}$ of $R$ dividing $A-\zeta$. It then follows by hypothesis that $\mathfrak{m}_{\zeta}$ also divides $B-\zeta'$ for some root of unity $\zeta'$ and so $(\zeta,\zeta') \in \mathcal{C}'(F) \subseteq \mathbb{G}^2_{m,F}(F) \simeq (F^{\ast})^2$. Varying $n$ shows that $\mathcal{C}'$ contains a Zariski dense set of special points. By the theorem of Ihara--Serre--Tate \cite{Lang_1965}, $\mathcal{C}'$ is a special subvariety of $\mathbb{G}^2_{m,F}$ and so $A,B$ are multiplicatively dependent. Since $A,B \not\in F$, it follows that there exist coprime non-zero integers $k, \ell \in \mathbb{Z}$ and a root of unity $\eta$ of order $r$ such that $A^kB^{\ell} = \eta$ holds in the fraction field of $R$.

Let $p$ denote a prime such that $p > \max\{N_0,|k|,|\ell|,r\}$ and let $\zeta$ be a root of unity such that $\xi := \eta\zeta^{-k}$ has order $|\ell| rp$. It follows that the order $m$ of $\zeta$ is divisible by $p$ and divides $|k\ell| rp$. By our assumptions on $N_0$ and $p$, there exists some maximal ideal $\mathfrak{m}_{\zeta}$ of $R$ dividing $A-\zeta$.

It then follows from the above construction that $\mathfrak{m}_{\zeta}$ also divides $(B^{|\ell|}-\xi)(B^{|\ell|}-\xi^{-1})$ and so divides $B-\xi'$ for some root of unity $\xi'$ of order $\ell^2 rp$. At the same time, $\mathfrak{m}_{\zeta}$ also divides $B^{m}-1$ by hypothesis. Since $m$ divides $|k\ell| rp$, $\mathfrak{m}_{\zeta}$ divides $B^{|k\ell| rp}-1$ and so divides $B-\zeta'$ for some root of unity $\zeta'$ of order dividing $|k \ell| rp$. It follows that $\zeta' = \xi'$, so $\ell$ divides $k$ and $\ell \in \{\pm1\}$ since $k$ and $\ell$ are coprime.

After taking the reciprocals of both sides of the equation $A^kB^{\ell} = \eta$ and replacing $(k,\eta)$ by $(-k,\eta^{-1})$ if necessary, we can assume without loss of generality that $\ell = -1$, \emph{i.e.}, that $B = \eta^{-1} A^k$ in the fraction field of $R$. Let again $p$ denote a prime such that $p > \max\{N_0,|k|,r\}$ and let $\zeta$ be a root of unity of order $p$. By our assumption on $N_0$, there exists some maximal ideal $\mathfrak{m}_{\zeta}$ of $R$ dividing $A-\zeta$. It follows that $\mathfrak{m}_{\zeta}$ divides $B-\eta^{-1}\zeta^k$ while also dividing $B^p-1$ by hypothesis. The order of $\eta^{-1}\zeta^k$ is $rp$ and so $r = 1$. Hence we have $\eta=1$ and we deduce that $B = A^k$.

In case ($1$), we are now done. In case ($2$), let again $p$ denote a prime such that $p > \max\{N_0,|k|\}$ and let $\zeta$ be a root of unity of order $p|k|$. By our assumption on $N_0$, there exists some maximal ideal $\mathfrak{m}_{\zeta}$ of $R$ dividing $A-\zeta$. It follows that $\mathfrak{m}_{\zeta}$ divides $B-\zeta^k$ while also dividing $B-\zeta'$ for some root of unity $\zeta'$ of order $p|k|$ by hypothesis. We deduce that $\zeta' = \zeta^k$, so $p|k| = p$ and $k \in \{\pm 1\}$.
\end{proof}

\section{The modular support problem}\label{sec:modularsupportproblem}

Let $R$ be a Dedekind domain and, for $D \in \mathbb{D}$, recall that $H_D(T)$ denotes the image under the canonical ring homomorphism $\mathbb{Z}[T] \to R[T]$ of the Hilbert class polynomial associated to an imaginary quadratic order of discriminant $D$. The \emph{modular support problem} for $R$ is the support problem for the family of polynomials $H_D(T) \in R[T]$ ($D \in \mathbb{D}$). We will consider the cases where $R$ is either the coordinate ring of a smooth affine irreducible curve over an algebraically closed field of characteristic $0$ or a ring of $S$-integers in some number field.

We start by proving an auxiliary proposition that will be used in both cases. It is inspired by the theory of isogeny volcanoes.

\begin{prop}\label{lem: no horizontal}
Let $F$ be an algebraically closed field. Let $E_{/F}$, $E'_{/F}$ be two elliptic curves with complex multiplication by the same imaginary quadratic order $\mathcal{O}$ and let $\varphi: E \to E'$ be an isogeny with cyclic kernel such that $\deg \varphi$ and $\mathrm{disc}(\mathcal{O})\max\{1,\mathrm{char}(F)\}$ are coprime. Then there does not exist any prime $\ell$ that divides $\deg \varphi$ and is inert in the fraction field of $\mathcal{O}$.
\end{prop}

\begin{proof}
Since $\gcd(\deg \varphi,\max\{1,\mathrm{char}(F)\}) = 1$, the isogeny $\varphi$ is separable. Let $P \in E(F)$ be a generator of $\ker \varphi$. Fix an abstract isomorphism $\iota: \End_{F}(E) \to \End_{F}(E')$ and set $\mathfrak{a} = \{\alpha \in \End_{F}(E); \alpha(P) = 0\}$. Then $\mathfrak{a}$ is an ideal of $\End_{F}(E)$ and we begin by showing that $[\End_{F}(E):\mathfrak{a}] = \deg \varphi$.

We have $[\End_{F}(E):\mathfrak{a}] = |\{\beta(P); \beta \in \End_{F}(E)\}|$. Let $\sigma: \End_{F}(E) \to \End_{F}(E)$ denote the unique non-trivial ring automorphism of $\End_{F}(E)$. One can check that for every endomorphism $\beta \in \End_{F}(E)$, there exists $\beta' \in \{\iota(\beta),\iota(\sigma(\beta))\}$ such that $\varphi \circ \beta = \beta' \circ \varphi$. Indeed, since $\varphi$ is separable, it follows from \cite[III, Corollary 4.11]{Silverman_book_2009} that there exists $\beta'' \in \End_{F}(E')$ such that $(\deg \varphi)(\varphi \circ \beta) = \beta'' \circ \varphi$. Hence, $Q(\beta'') \circ \varphi = \varphi \circ Q((\deg \varphi)\beta) = 0$, where $Q$ denotes the minimal polynomial of $(\deg \varphi)\beta$ in $\mathbb{Z}[T]$, so $Q(\beta'') = 0$ and $\beta'' \in \{(\deg \varphi)\iota(\beta),(\deg \varphi)\iota(\sigma(\beta))\}$. It follows that $\varphi \circ \beta = \beta' \circ \varphi$ for some $\beta' \in \{\iota(\beta),\iota(\sigma(\beta))\}$ as desired.

This implies that $\varphi(\beta(P)) = \beta'(\varphi(P)) = 0$, so $\beta(P) \in \ker \varphi$. Since $\ker \varphi = \mathbb{Z} \cdot P$, we deduce that $|\{\beta(P); \beta \in \End_{F}(E)\}| = \deg \varphi$, so $[\End_{F}(E):\mathfrak{a}] = \deg \varphi$. Since $\gcd(\deg \varphi,\mathrm{disc}(\mathcal{O})) = 1$, this implies that $\mathfrak{a}$ is invertible (see \cite[Proposition 7.4 and Lemma 7.18]{Cox_book_2013}).

Aiming for a contradiction, we now assume that there exists a prime $\ell$ that divides $\deg \varphi$ and is inert in the fraction field of $\mathcal{O}$. Since $\gcd(\deg \varphi,\mathrm{disc}(\mathcal{O})) = 1$, it follows from \cite[Proposition 7.20]{Cox_book_2013} that there exists an ideal $\mathfrak{b}$ of $\End_{F}(E)$ such that $\mathfrak{a} = \ell\mathfrak{b}$ and
\[ \deg \varphi = [\End_{F}(E):\mathfrak{a}] = \ell^2[\End_{F}(E):\mathfrak{b}].\]

Hence, $\mathfrak{b} = \{\alpha \in \End_{F}(E); \ell\alpha \in \mathfrak{a}\} = \{\alpha \in \End_{F}(E); \alpha(\ell P) = 0\}$. By \cite[III, Proposition 4.12]{Silverman_book_2009}, there exist an elliptic curve $E''_{/F}$ and a separable isogeny $\psi: E \to E''$ whose kernel is generated by $\ell P$. We have
\[\deg \psi = |\mathbb{Z} \cdot \ell P| \leq |\{\beta(\ell P); \beta \in \End_{F}(E)\}| = [\End_{F}(E):\mathfrak{b}] = \ell^{-2} \deg \varphi\]
or equivalently $(\deg \varphi)/(\deg \psi) \geq \ell^2$.

Since $\ker \psi \subseteq \ker \varphi$ and $\psi$ is separable, \cite[III, Corollary 4.11]{Silverman_book_2009} implies that there exists an isogeny $\xi: E'' \to E'$ such that $\varphi = \xi \circ \psi$. We have $\deg \varphi = (\deg \xi)(\deg \psi)$, so in particular, $\deg \xi$ and $\max\{1,\mathrm{char}(F)\}$ are coprime and $\xi$ is  therefore separable. By \cite[III, Theorem 4.10~(c)]{Silverman_book_2009} together with the above, we have $|\ker \xi| = \deg \xi = (\deg \varphi)/(\deg \psi) \geq \ell^2$. But $\ker \xi = \psi(\ker \varphi) \simeq (\ker \varphi)/(\ker \psi) \simeq \mathbb{Z}/\ell \mathbb{Z}$ and we get the desired contradiction.
\end{proof}

\subsection{The modular support problem over function fields of characteristic $0$}\label{sec:modularsupportproblemfunctionfield}

In this section, we consider the case where $R$ is the coordinate ring of a smooth affine irreducible curve over an algebraically closed field $F$ of characteristic $0$. Since the Hilbert class polynomials are irreducible over $\mathbb{Q}$, monic, and pairwise distinct, it suffices, thanks to Example \ref{ex:support_1}, to consider the problem for $A,B \in R\backslash F$. The following theorem gives a complete solution to the modular support problem in this setting.

\begin{thm} \label{thm: support problem in function fields}
Let $R$ be the coordinate ring of a smooth affine irreducible curve $\mathcal{C}$ over an algebraically closed field $F$ of characteristic $0$. Let $A, B \in R \backslash F$ and suppose that there exists $D_0 \in \mathbb{N}$ such that for all discriminants $D \in \mathbb{D}$ with $|D|>D_0$ every prime ideal of $R$ that divides $H_D(A)$ also divides $H_D(B)$. Then $A=B$.
\end{thm}

\begin{proof}
We first show that $\Phi_N(A,B)=0$ for some $N \in \mathbb{N}$, where we recall that $\Phi_N$ denotes the $N$-th modular polynomial. Suppose by contradiction that this is not the case. Then Theorem \ref{thm: gcd in function fields} yields a non-zero ideal $J$ of $R$ such that 
\[
\gcd(H_D(A),H_D(B)) \mid J
\]
for all $D \in \mathbb{D}$. By our assumption, every prime ideal of $R$ that divides some $H_D(A)$ ($D\in \mathbb{D}$ with $|D| > D_0$) also divides the ideal $J$. On the other hand, since $H_D(A)$ and $H_{D'}(A)$ are coprime for $D, D' \in \mathbb{D}$ with $D \neq D'$, we deduce from \cite[Proposition 15.4~(1)]{GoertzWedhorn} applied to the morphism $\mathcal{C} \to \mathbb{A}^1_F = Y(1)_F$ induced by $A$ that for infinitely many $D \in \mathbb{D}$, there exists a maximal ideal $\mathfrak{m}_D$ that divides $H_D(A)$. So $J$ must be divisible by infinitely many pairwise distinct maximal ideals, which contradicts the fact that $J \neq (0)$. Hence there exists $N \in \mathbb{N}$ such that $\Phi_N(A,B)=0$.

Our goal is now to show that $N=1$. This would yield the desired result, since $\Phi_1(A,B)=A-B$. Assume then, again by contradiction, that $N>1$ and let $p \in \mathbb{N}$ be a prime factor of $N$.

The morphisms $\psi_A, \psi_B: \mathcal{C} \to Y(1)_F$ induced by $A,B$ are non-constant, so in particular, by \cite[Proposition 15.4~(1)]{GoertzWedhorn}, the set $Y(1)_F\backslash \psi_A(\mathcal{C})$ is finite. By enlarging $D_0$ if necessary, we can assume without loss of generality that $\sigma \in \psi_A(\mathcal{C})$ for every singular modulus $\sigma$ whose discriminant $D$ satifies $|D| > D_0$. By Dirichlet's theorem on primes in arithmetic progressions, there exists a fundamental discriminant $D \in \mathbb{D}$ with $|D|>D_0$ such that $\gcd(D,N) = 1$ and $p$ is inert in $K:=\mathbb{Q}(\sqrt{D}) \subseteq \overline{\mathbb{Q}}$.

By our assumption on $D_0$, there exist an elliptic curve $E_{/F}$ and a point $P \in \mathcal{C}(F)$ such that $\End_F(E) \simeq \mathcal{O}_K$ and $\psi_A(P)=j(E) \in Y(1)(F) \simeq F$ is the $j$-invariant of $E$. Let $\mathfrak{m}_P$ denote the maximal ideal of $R$ corresponding to $P$. Clearly $\mathfrak{m}_P \mid H_D(A)$.

By assumption, $\mathfrak{m}_P$ divides $H_D(B)$ and so $\psi_B(P)$ is the $j$-invariant of an elliptic curve $E'_{/F}$ with complex multiplication by $\mathcal{O}_K$. We know that $\Phi_N(\psi_A(P),\psi_B(P))=0$, so, by \cite[Proposition 14.11]{Cox_book_2013}, there exists an isogeny $\varphi: E \to E'$ of degree $N$ with cyclic kernel. This contradicts Proposition \ref{lem: no horizontal} with $\ell = p$. We conclude that $N=1$ and the theorem follows.
\end{proof}

\subsection{The modular support problem over number fields}\label{sec:modularsupportproblemnumberfield}

In this section, we analyze the modular support problem over a ring of $S$-integers in a number field. In all but finitely many exceptional cases, we will be able to show that the conclusion of Theorem \ref{thm: support problem in function fields} stays true. However, the proof is much more involved. We first prove an auxiliary lemma that is a modular analogue of Lemma \ref{lem:order_modulo_prime}.

\begin{lem}\label{lem:CM_modulo_prime}
Let $K$ be a number field, let $S$ be a finite set of maximal ideals of $\mathcal{O}_K$, and let $\mathfrak{p}$ denote a maximal ideal of $\mathcal{O}_{K,S}$ of residue characteristic $p$. The following are equivalent for $a \in \mathcal{O}_{K,S}$ and a discriminant $D \in \mathbb{D}$:
\begin{enumerate}
    \item there exist a finite field extension $K \subseteq L$, a singular modulus $j \in L$ of discriminant $D$, and a prime $\mathfrak{P}$ of $L$ lying above $\mathfrak{p}$ such that $a \equiv j \mbox{ \emph{mod} } \mathfrak{P}$,
    \item $H_{Dp^{2\ell}}(a) \in \mathfrak{p}$ for some $\ell \in \mathbb{Z}_{\geq 0}$, and
    \item $H_{Dp^{2\ell}}(a) \in \mathfrak{p}$ for all $\ell \in \mathbb{Z}_{\geq 0}$.
\end{enumerate}
\end{lem}

\begin{proof}
Choosing $\ell=0$, we see immediately that (1) implies (2) and that (3) implies (1) with $L$ equal to a splitting field of $H_D$ over $K$. The equivalence of (2) and (3) follows from Proposition \ref{prop:Hilbert_class_polynomial_mod_p}.
\end{proof}

We are now ready to prove Theorem \ref{thm:modular_support_problem intro} that we restate here for the reader's convenience.

\begin{thm}\label{thm:modular_support_problem}
Let $K$ be a number field and let $S$ be a finite set of maximal ideals of $\mathcal{O}_K$. Let $j, j' \in \mathcal{O}_{K,S}$. Suppose that there exists $D_0 \in \mathbb{N}$ such that all the prime ideals of $\mathcal{O}_{K,S}$ dividing $H_D(j)$ also divide $H_D(j')$ for every $D \in \mathbb{D}$ with $|D| > D_0$. Then either $j = j'$ or there exists $\widetilde{D} \in \mathbb{D}$ such that $H_{\widetilde{D}}(j)=H_{\widetilde{D}}(j')=0$.
\end{thm}

\begin{proof}
Denote by $E_{j}$, $E_{j'}$ any two fixed elliptic curves over $K$ with $j$-invariants $j$, $j'$ respectively. We begin by showing that, under the hypothesis of the theorem, the two curves are geometrically isogenous.

Fix an algebraic closure $\overline{K}$ of $K$. Let $\mathfrak{p} \subseteq \mathcal{O}_{K,S}$ be a prime ideal of good reduction for $E_j$ and $E_{j'}$ and let $\mathfrak{P}$ be a prime of $\overline{K}$ lying above it. By Deuring's lifting theorem \cite[Chapter 13, Theorem 14]{Lang_1987}, there exist a discriminant $D \in \mathbb{D}$ and an elliptic curve $(E_D)_{/\overline{K}}$ with complex multiplication by the order of discriminant $D$ such that $E_D \text{ mod } \mathfrak{P} \simeq E_j \text{ mod } \mathfrak{P}$. In particular, $\mathfrak{p} \mid H_D(j)$ so that by Lemma \ref{lem:CM_modulo_prime} $\mathfrak{p} \mid H_{Dp^{2n}}(j)$ for every $n \in \mathbb{N}$, where $p$ denotes the residue characteristic of $\mathfrak{p}$. Choosing $n$ large enough, we deduce from the hypothesis of the theorem that $\mathfrak{p} \mid H_{Dp^{2n}}(j')$ and then, again by Lemma \ref{lem:CM_modulo_prime}, that $\mathfrak{p} \mid H_{D}(j')$. Hence, there exists an elliptic curve $(E'_D)_{/\overline{K}}$ with complex multiplication by the order of discriminant $D$ such that $E'_D \text{ mod } \mathfrak{P} \simeq E_{j'} \text{ mod } \mathfrak{P}$. Since $E_D$ and $E'_D$ are isogenous, the same must be true for $E_j \text{ mod } \mathfrak{P}$ and $E_{j'} \text{ mod } \mathfrak{P}$. Hence, for all but finitely many maximal ideals $\mathfrak{p}$ of $\mathcal{O}_{K,S}$, the reduced elliptic curves $E_j$ and $E_{j'}$ modulo $\mathfrak{p}$ are geometrically isogenous. By \cite[Theorem 1]{KhareLarsen_Preprint} we conclude that $(E_j)_{\overline{K}}$ and $(E_{j'})_{\overline{K}}$ are isogenous. In particular, the geometric endomorphism rings of $E_j$ and $E_{j'}$ must have the same $\mathbb{Z}$-rank. 

It now suffices to consider two cases.

\textbf{Case 1:} Both $E_j$ and $E_{j'}$ do not have complex multiplication.

Let $\varphi: (E_j)_{\overline{K}} \to (E_{j'})_{\overline{K}}$ be an isogeny and set $d = \deg \varphi$. We can assume without loss of generality that $\varphi$ has cyclic kernel. If $d=1$, then $\varphi$ is an isomorphism and $j=j'$. Suppose then by contradiction that $d>1$ and hence $j \neq j'$ since $E_{j}$ does not have complex multiplication. We want to apply Theorem \ref{thm:Zarhin} with the following inputs:
\begin{itemize}
    \item We take $\mathcal{P}$ to be the set of rational primes dividing $dN(\mathfrak{K})$, where $\mathfrak{K} \subseteq \mathcal{O}_{K,S}$ denotes the ideal generated by $\prod_{D \in \mathbb{D},\hspace{2pt}|D|\leq D_0}{H_D(j)}$. Note that $\mathcal{P}$ is finite since $j$ is not a singular modulus, and $\mathcal{P} \neq \emptyset$ because $d>1$ by assumption.
    \item We take $L$ to be an imaginary quadratic field where all primes $\ell \in \mathcal{P}$ are inert. Such a field certainly exists by Dirichlet's theorem on primes in arithmetic progressions. We also take $\mathcal{O}=\mathcal{O}_L$ to be the ring of integers in $L$.
\end{itemize}

Since $E_j$ does not have complex multiplication and $E_j$ and $E_{j'}$ are geometrically isogenous, by Theorem \ref{thm:Zarhin} there exist infinitely many maximal ideals $\mathfrak{p}\subseteq \mathcal{O}_{K,S}$ such that $\mathfrak{p} \nmid d\mathfrak{K}$, $E_j$ and $E_{j'}$ have good ordinary reduction modulo $\mathfrak{p}$, and 
\begin{equation} \label{eq:tensor_in_Zarhin_theorem}
    \End_{k_{\mathfrak{p}}}(E_j \text{ mod } \mathfrak{p}) \otimes_{\mathbb{Z}} \mathbb{Z}_\ell \simeq \mathcal{O} \otimes_{\mathbb{Z}} \mathbb{Z}_\ell=:\mathcal{O}_\ell
\end{equation}
for all $\ell \in \mathcal{P}$. In particular, denoting by $\mathcal{A}$ this infinite set of primes obtained from Theorem \ref{thm:Zarhin} and setting $R_\mathfrak{p}:=\End_{k_{\mathfrak{p}}}(E_j \text{ mod } \mathfrak{p})$ for all primes $\mathfrak{p} \in \mathcal{A}$, we deduce from our choice of $L$ and from \eqref{eq:tensor_in_Zarhin_theorem} that all primes $\ell \in \mathcal{P}$ are inert in the fraction field $L_\mathfrak{p}$ of $R_\mathfrak{p}$ for every $\mathfrak{p} \in \mathcal{A}$. Moreover, \eqref{eq:tensor_in_Zarhin_theorem} also implies that the conductor of $R_\mathfrak{p}$ is not divisible by any $\ell \in \mathcal{P}$ since, by \cite[Equation (1.8) on p. 109]{Froehlich_Taylor_book}, $\mathcal{O}_\ell$ is the ring of integers of the local field $L \otimes_{\mathbb{Z}} \mathbb{Z}_\ell$, which is isomorphic to a quadratic extension of $\mathbb{Q}_{\ell}$. Finally, we remark that $R_\mathfrak{p}$ can be identified with $\End_{\overline{k}_\mathfrak{p}} (E_j \text{ mod } \mathfrak{p})$ by Lemma \ref{lem:endomorphisms_ordinary_elliptic_curve}, where $\overline{k}_\mathfrak{p}$ denotes a fixed algebraic closure of $k_{\mathfrak{p}}$.

Fix now $\mathfrak{p} \in \mathcal{A}$ and set $\widetilde{E}_j:=E_j \text{ mod } \mathfrak{p}$. By Deuring's lifting theorem \cite[Chapter 13, Theorem 14]{Lang_1987}, there exist an elliptic curve $E_{/\overline{K}}$ with complex multiplication by an imaginary quadratic order of some discriminant $D$ and a prime $\mathfrak{P} \subseteq \overline{K}$ lying above $\mathfrak{p}$ such that $E \text{ mod } \mathfrak{P} \simeq (\widetilde{E}_j)_{\overline{k}_\mathfrak{p}}$, where we identify the residue field of $\mathfrak{P}$ and $\overline{k}_\mathfrak{p}$ via a fixed isomorphism. In particular, $\mathfrak{p} \mid H_D(j)$. Since by construction $\mathfrak{p} \nmid \mathfrak{K}$, we must have $|D| > D_0$ so that $\mathfrak{p} \mid H_D(j')$ by hypothesis. This implies that $\widetilde{E}_{j'}:=E_{j'} \text{ mod } \mathfrak{p}$ and $\widetilde{E}_{j}$ have complex multiplication by the same imaginary quadratic order $R_\mathfrak{p}$ by \cite[Chapter 13, Theorem 12]{Lang_1987}. Moreover, by \cite[II, Proposition 4.4]{Silverman_book_1994} the reduction of $\varphi$ modulo $\mathfrak{P}$ gives an isogeny $\widetilde{\varphi}: (\widetilde{E}_{j})_{\overline{k}_{\mathfrak{p}}} \to (\widetilde{E}_{j'})_{\overline{k}_{\mathfrak{p}}}$ of degree $d$. Since the residue characteristic $p$ of $\mathfrak{p}$ does not divide $d$, the kernel of $\widetilde{\varphi}$ must be equal to the reduction of $\ker \varphi$ modulo $\mathfrak{P}$ by \cite[III, Theorem 4.10~(c), and VII, Proposition 3.1~(b)]{Silverman_book_2009}. Hence, $\widetilde{\varphi}$ has cyclic kernel. Since every prime dividing $d$ does not divide $\mathrm{disc}(R_{\mathfrak{p}})p$ and is inert in $L_{\mathfrak{p}}$ by construction, but $d > 1$, we can apply Proposition \ref{lem: no horizontal} to deduce a contradiction. Hence $j=j'$ and this concludes the proof in this case.

\textbf{Case 2:} Both $E_j$ and $E_{j'}$ have complex multiplication.

Let $D_j \in \mathbb{D}$ be the discriminant of $\End_{\overline{K}}(E_j)$. Take $p \in \mathbb{N}$ to be a prime such that $\left( \frac{D_j}{p} \right)=1$ and let $E_{/\overline{K}}$ be an elliptic curve with complex multiplication by the imaginary quadratic order of discriminant $p^{2n} D_j$ for some fixed $n \in \mathbb{N}$. By \cite[Chapter 13, Theorem 12]{Lang_1987}, for every prime $\mathfrak{P} \subseteq \overline{K}$ lying above $p$ the reduction of $E$ modulo $\mathfrak{P}$ is an elliptic curve with complex multiplication by the imaginary quadratic order of discriminant $D_j$ and by Proposition \ref{prop:Hilbert_class_polynomial_mod_p} we can choose $E$ in such a way that $\mathfrak{P} \mid (j-j(E))$ and hence $\mathfrak{p} \mid H_{p^{2n} D_j}(j)$.

If we now choose $n$ sufficiently large, the hypothesis of the theorem implies that we also have that $\mathfrak{p} \mid H_{p^{2n} D_j}(j')$. Again, it follows from \cite[Chapter 13, Theorem 12]{Lang_1987} that $E_{j'}$ has complex multiplication by an order of discriminant $D_{j'}=p^{2k} D_j$ for some $k \in \mathbb{Z}_{\geq 0}$. Arguing in the same way with a prime $\ell \neq p$ such that $\left( \frac{D_j}{\ell} \right)=1$, we conclude that there exists $r \in \mathbb{Z}_{\geq 0}$ such that
\[
D_{j'}=p^{2k} D_j= \ell^{2r} D_j
\]
and, since $\gcd(D_j, p\ell)=1$, we deduce that $D_{j'} = D_j$. In particular, we have $H_{D_j}(j)=H_{D_j}(j')=H_{D_{j'}}(j')=0$ as desired.
\end{proof}

\begin{rmk}
In the proof of Theorem \ref{thm: support problem in function fields}, we knew, for algebro-geometric reasons, that all but finitely many imaginary quadratic orders could be obtained as the geometric endomorphism rings of specializations of the elliptic curve with $j$-invariant $A$. For the proof of Theorem \ref{thm:modular_support_problem}, we only have Theorem \ref{thm:Zarhin} by Zarhin, which yields a weaker analogue of this statement. However, this analogous statement is still strong enough to allow us to fabricate a situation where Proposition \ref{lem: no horizontal} applies.
\end{rmk}

In the case where $H_{\widetilde{D}}(j) = H_{\widetilde{D}}(j') = 0$ for some $\widetilde{D} \in \mathbb{D}$, it is unclear whether the support property can be satisfied for all but finitely many $D \in \mathbb{D}$. Numerical experiments seem to suggest that this should be true only if $j = j'$. On the other hand, we will now present an example of two distinct Galois conjugate singular moduli for which the support property holds for $25\%$ of all $D \in \mathbb{D}$. In particular, if the support property holds for infinitely many $D \in \mathbb{D}$, this does not imply that $j = j'$.

\begin{thm}\label{thm:the lonely theorem}
Let 
\[
j_1=\frac{-191025-85995\sqrt{5}}{2} \quad \mbox{ and } \quad j_2=\frac{-191025+85995\sqrt{5}}{2}
\]
be the two singular moduli of discriminant $-15$ in $\overline{\mathbb{Q}}$. Then for every discriminant $D\in \mathbb{D}$ with $D\equiv 1 \mbox{ \emph{mod} } 8$, the support property holds in both directions, \emph{i.e.}, for every maximal ideal $\mathfrak{p}$ of $\mathbb{Z}\left[\frac{-1+\sqrt{5}}{2}\right]$ we have 
\[
\mathfrak{p} \mid H_D(j_1) \quad \Longleftrightarrow \quad \mathfrak{p} \mid H_D(j_2).
\]
\end{thm}

\begin{proof}
In this proof, we will use the notion of two isogenies being equivalent, which we now recall. Let $F$ be an arbitrary algebraically closed field. If $E_1, E_2$, and $E_3$ are three elliptic curves over $F$ and if $\varphi:E_1 \to E_2$ and $\psi:E_1 \to E_3$ are two isogenies, we say that $\varphi$ and $\psi$ are \textit{equivalent} if there exists an isomorphism $\xi:E_2 \to E_3$ such that $\xi \circ \varphi=\psi$. If $\deg \varphi$ and $\deg \psi$ are not divisible by $\text{char}(F)$ and so $\varphi$ and $\psi$ are separable, then this is the same as requiring that $\ker \varphi= \ker \psi$ by \cite[III, Corollary 4.11]{Silverman_book_2009}. Of course, if $\varphi$ and $\psi$ are equivalent, then $j(E_2)=j(E_3)$.

We fix an embedding $\overline{\mathbb{Q}}\hookrightarrow \mathbb{C}$. Let $E_{/\mathbb{C}}$ be an elliptic curve with complex multiplication by an imaginary quadratic order $\mathcal{O} \subseteq \overline{\mathbb{Q}}$ of discriminant $D \equiv 1 \text{ mod } 8$ and fix an isomorphism $[\cdot]_E: \mathcal{O} \to \End_{\mathbb{C}}(E)$. The hypothesis that $D\equiv 1 \text{ mod } 8$ implies that $2\mathcal{O}=\mathfrak{p}_2 \mathfrak{p}_2'$ where $\mathfrak{p}_2, \mathfrak{p}_2' \subseteq \mathcal{O}$ are distinct Galois conjugate invertible ideals of norm $2$ (see  \cite[Proposition 7.20]{Cox_book_2013}). We now apply \cite[II, Proposition 1.2~(a)~(ii)]{Silverman_book_1994} with $\mathfrak{a} = \mathfrak{p}_2^{-1}$ to deduce that the endomorphism ring of the quotient $(E_{\mathfrak{p}_2})_{/\mathbb{C}}$ of $E$ by the $\mathfrak{p}_2$-torsion subgroup 
\[
E[\mathfrak{p}_2]:=\{P\in E(\mathbb{C}); [\alpha]_E(P)=0 \text{ for all } \alpha \in \mathfrak{p}_2 \}
\]
is isomorphic to $\mathcal{O}$. For this, note that \cite[II, Proposition 1.2~(a)~(ii)]{Silverman_book_1994}, although formulated only for maximal orders, also holds for an arbitrary order (with the same proof) as long as the fractional ideal $\mathfrak{a}$ is invertible. Let $\varphi: E \to E_{\mathfrak{p}_2}$ denote an isogeny with kernel $E[\mathfrak{p}_2]$. The degree of $\varphi$ is equal to the norm of $\mathfrak{p}_2$, which is $2$. Similarly, there exists an isogeny $\varphi': E \to E_{\mathfrak{p}'_2}$ of degree $2$ with kernel $E[\mathfrak{p}_2']$ such that $(E_{\mathfrak{p}'_2})_{/\mathbb{C}}$ is an elliptic curve with complex multiplication by $\mathcal{O}$ where $E[\mathfrak{p}_2']$ is defined analogously to $E[\mathfrak{p}_2]$. The isogenies $\varphi, \varphi'$ are not equivalent since $\mathfrak{p}_2 + \mathfrak{p}'_2 = \mathcal{O}$ and therefore $E[\mathfrak{p}_2]\neq E[\mathfrak{p}_2']$. This shows in particular that every elliptic curve with complex multiplication by an imaginary quadratic order of discriminant $D\equiv 1 \text{ mod } 8$ admits two nonequivalent isogenies of degree $2$ towards elliptic curves with complex multiplication by the same order.

In the case $D=-15$, these isogenies can be described explicitly in complex analytic terms. For $k\in \{1,2\}$, denote by $(E_k)_{/\overline{\mathbb{Q}}}$ a fixed elliptic curve such that $j(E_k)=j_k$. Since $2\mathcal{O}= \mathfrak{p}_2 \mathfrak{p}'_2$ where $\mathfrak{p}_2=\left(2, \frac{1+\sqrt{-15}}{2} \right)$ is non-principal, it follows from \cite[II, Proposition 1.2]{Silverman_book_1994} that there exist complex analytic isomorphisms
\[
\xi_1: E_i(\mathbb{C}) \to \mathbb{C}/\mathcal{O}, \hspace{0.5cm} \xi_2: E_k(\mathbb{C}) \to \mathbb{C}/\mathfrak{p}_2, \quad\{i,k\} = \{1,2\}.
\]
We then see that the map $z\mapsto 2z$ induces an isogeny $\mathbb{C}/\mathcal{O} \to \mathbb{C}/\mathfrak{p}_2$ with kernel $\mathfrak{p}_2'^{-1}/\mathcal{O}$ and an isogeny $\mathbb{C}/\mathcal{O} \to \mathbb{C}/\mathfrak{p}'_2$ with kernel $\mathfrak{p}_2^{-1}/\mathcal{O}$. Since $\mathfrak{p}_2$ and $\mathfrak{p}'_2$ are in the same ideal class in $\mathrm{Pic}(\mathcal{O})$, we have $\mathbb{C}/\mathfrak{p}_2 \simeq \mathbb{C}/\mathfrak{p}'_2$ and we thus obtain two inequivalent isogenies $\varphi_1, \varphi_2: (E_i)_{\mathbb{C}} \to (E_k)_{\mathbb{C}}$. Since $E_1$ and $E_2$ are defined over $\overline{\mathbb{Q}}$, both these isogenies are base changes of isogenies $E_i \to E_k$ that we will denote by $\varphi_1$ and $\varphi_2$ as well. We can assume without loss of generality, after maybe applying an automorphism of $\overline{\mathbb{Q}}$ which maps $\sqrt{5}$ to $-\sqrt{5}$, that $i = 1$ and $k = 2$.

We are now ready to conclude the proof of the theorem. Since $j_1$ and $j_2$ are Galois conjugate, it suffices to prove that $\mathfrak{p} \mid H_D(j_1)$ implies that $\mathfrak{p} \mid H_D(j_2)$ for every maximal ideal $\mathfrak{p}$ of $\mathbb{Z}\left[\frac{-1+\sqrt{5}}{2}\right]$. Fix a negative discriminant $D\equiv 1 \text{ mod } 8$, a maximal ideal $\mathfrak{p}$ of $\mathbb{Z}\left[\frac{-1+\sqrt{5}}{2}\right]$ such that $\mathfrak{p} \mid H_D(j_1)$, and a prime $\mathfrak{P}\subseteq \overline{\mathbb{Q}}$ lying above $\mathfrak{p}$. Since $E_1$ and $E_2$ have complex multiplication, they both have good reduction at $\mathfrak{P}$. As a first step, assume that $\mathfrak{p}$ lies above $2$. Since $2$ is inert in $\mathbb{Q}(\sqrt{5})$ and $j_1,j_2$ are Galois conjugate, we deduce that $\mathfrak{p} = (2)$ and $\mathfrak{p} \mid H_D(j_1) \Leftrightarrow \mathfrak{p} \mid H_D(j_2)$.

Assume from now on that $\mathfrak{p}$ does not lie above $2$ and that $\mathfrak{p}$ divides $H_D(j_1)$. To ease notation, we will use a tilde $\widetilde{\cdot}$ to denote the reduction of some object (curve, morphism, etc.) modulo $\mathfrak{P}$. It follows from our assumption that there exists $j\in \overline{\mathbb{Q}}$ such that $H_D(j)=0$ and $j_1 \equiv j \text{ mod } \mathfrak{P}$. Fix an elliptic curve $(E_j)_{/\overline{\mathbb{Q}}}$ with $j(E_j)=j$. The above discussion shows that $E_j$ admits at least two non-equivalent degree-$2$ isogenies $\psi_k: E_j \to E_{D,k}$ for $k\in \{1,2\}$, where $\End_{\overline{\mathbb{Q}}}(E_{D,k}) \simeq \mathbb{Z}\left[ \frac{D+\sqrt{D}}{2}\right]$. We set $j_{D,k}:=j(E_{D,k})$ so that $H_D(j_{D,k})=0$ for $k \in \{1,2\}$ (it is possible that $j_{D,1}=j_{D,2}$, \emph{e.g.}, if $D = -15$). Since the aforementioned elliptic curves all have complex multiplication, they all have good reduction at $\mathfrak{P}$. Let $E_j[2]$ denote the $2$-torsion subgroup of $E_j$, then the reduction map $E_j[2] \to \widetilde{E_j[2]}$ is injective by \cite[VII, Proposition 3.1~(b)]{Silverman_book_2009}. It follows that, for $k \in \{1,2\}$, the kernel of $\widetilde{\psi}_k$ is precisely the reduction modulo $\mathfrak{P}$ of $\ker{\psi_k} \subseteq E_j[2]$ since $\deg \widetilde{\psi}_k = \deg \psi_k = 2$ by \cite[II, Proposition 4.4]{Silverman_book_1994}. Hence, $\ker{\widetilde{\psi}_1} \neq \ker{\widetilde{\psi}_2}$ and so the reduced isogenies $\widetilde{\psi}_1, \widetilde{\psi}_2$ are nonequivalent. Similarly, it follows that also the reduced isogenies $\widetilde{\varphi_1}$ and $\widetilde{\varphi_2}$ are nonequivalent. On the other hand, the elliptic curve $\widetilde{E}_1 \simeq \widetilde{E}_j$ cannot possess more than three nonequivalent isogenies of degree $2$ since there are only three distinct subgroups of order $2$ of the $2$-torsion subgroup of $\widetilde{E_1}$. Hence, there exist $i,k \in \{1,2\}$ such that the isogeny $\widetilde{\varphi}_i$ is equivalent to the isogeny $\widetilde{\psi}_k$. This in particular implies that $\widetilde{j_2} = \widetilde{j}_{D,k}$ and we conclude that $\mathfrak{p}$ divides $H_D(j_2)$ as desired. 
\end{proof}

\section*{Acknowledgements}
We thank Jorge Mello, Pieter Moree, and Jonathan Pila for useful discussions and suggestions. The first author is grateful to the Max-Planck-Institut f\"ur Mathematik  in Bonn for its hospitality and financial support. He is also supported by ANR-20-CE40-0003 Jinvariant. The second author was supported by the Swiss National Science Foundation through the Early Postdoc.Mobility grant no. P2BSP2\_195703. He thanks the Mathematical Institute of the University of Oxford and his host there, Jonathan Pila, for hosting him as a visitor for the duration of this grant. He also thanks the first author for his hospitality.

\vspace{\baselineskip}
\noindent
\framebox[\textwidth]{
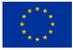
\begin{tabular*}{0.96\textwidth}{@{\extracolsep{\fill} }cp{0.84\textwidth}}
\raisebox{-0.7\height}{%
    \begin{tikzpicture}[y=0.80pt, x=0.8pt, yscale=-1, inner sep=0pt, outer sep=0pt, 
    scale=0.12]
    \definecolor{c003399}{RGB}{0,51,153}
    \definecolor{cffcc00}{RGB}{255,204,0}
    \begin{scope}[shift={(0,-872.36218)}]
      \path[shift={(0,872.36218)},fill=c003399,nonzero rule] (0.0000,0.0000) rectangle (270.0000,180.0000);
      \foreach \myshift in 
           {(0,812.36218), (0,932.36218), 
    		(60.0,872.36218), (-60.0,872.36218), 
    		(30.0,820.36218), (-30.0,820.36218),
    		(30.0,924.36218), (-30.0,924.36218),
    		(-52.0,842.36218), (52.0,842.36218), 
    		(52.0,902.36218), (-52.0,902.36218)}
        \path[shift=\myshift,fill=cffcc00,nonzero rule] (135.0000,80.0000) -- (137.2453,86.9096) -- (144.5106,86.9098) -- (138.6330,91.1804) -- (140.8778,98.0902) -- (135.0000,93.8200) -- (129.1222,98.0902) -- (131.3670,91.1804) -- (125.4894,86.9098) -- (132.7547,86.9096) -- cycle;
    \end{scope}
    \end{tikzpicture}%
}
&
Gabriel Dill has received funding from the European Research Council (ERC) under the European Union’s Horizon 2020 research and innovation programme (grant agreement n$^\circ$ 945714).
\end{tabular*}
}

\printbibliography

\end{document}